\documentclass[a4paper,reqno, 10pt]{amsart}

\usepackage{amsmath,amssymb,amsfonts,amsthm,mathrsfs}

\usepackage{verbatim,wasysym,cite}

\usepackage[latin1]{inputenc}
\usepackage{microtype}
\usepackage{color,enumitem,graphicx}
\usepackage[colorlinks=true,urlcolor=blue, citecolor=red,linkcolor=blue,
linktocpage,pdfpagelabels, bookmarksnumbered, bookmarksopen]{hyperref}
\usepackage[english]{babel}
\usepackage[symbol]{footmisc}
\renewcommand{\epsilon}{{\varepsilon}}
\usepackage{enumitem}

\numberwithin{equation}{section}
\newtheorem{theorem}{Theorem}[section]
\newtheorem{lemma}[theorem]{Lemma}
\newtheorem{remark}[theorem]{Remark}

\newtheorem{proposition}[theorem]{Proposition}
\newtheorem{corollary}[theorem]{Corollary}
\newtheorem{claim}[theorem]{Claim}

\def\leq{\leqslant}
\def\geq{\geqslant}
\newcommand{\C}{\mathbb C}

\newcommand{\R}{\mathbb R}
\newcommand{\N}{\mathbb N}

\def\({\left(}
\def\){\right)}
\def\<{\left\langle}
\def\>{\right\rangle}

\def\Sch{{\mathcal S}}
\def\A{\mathcal A}

\def\B{\mathcal B}
\def\F{\mathcal F}
\def\K{\mathcal K}
\def\TT{\mathcal T}
\def\Q{\mathcal Q}

\def\G{\mathcal G}
\def\L{\mathcal L}
\def\EE{\mathcal E}

\def\kk{\kappa}

\def\M{\mathcal M}
\def\Lm{\rm L}

\DeclareMathOperator{\RE}{Re}
\DeclareMathOperator{\IM}{Im}

\newcommand{\qtq}[1]{\quad\text{#1}\quad}

\newcommand{\ve}[1]{\textbf{#1}}

\begin{document}

\title[Dynamic of threshold solutions]{Threshold solutions for the energy-critical NLS system with quadratic interaction}

\author[Alex H. Ardila]{Alex H. Ardila}
\address{Department of Mathematics, Universidad del Valle, Colombia} 
\email{ardila@impa.br}

\author[Liliana Cely]{Liliana Cely}
\address{Escuela de Matem\'aticas, Universidad Industrial de Santander}
\email{mlcelpri@uis.edu.do}

\author[Fanfei Meng]{Fanfei Meng}
\address{Qiyuan Lab, Tsinghua University}
\email{mengfanfei17@gscaep.ac.cn}

\begin{abstract} 
In this paper, we study the Cauchy problem for a quadratic nonlinear Schr\"{o}dinger system in dimension six. In~\cite{GaoMengXuZheng}, the authors classified the behavior of solutions under the energy constraint $ E(\ve{u}) < E(\Q)$, where $ \Q$ denotes the ground state. In this work, we classify the dynamics of radial solutions at the threshold energy $E(\ve{u}) = E(\Q)$.
\end{abstract}


\maketitle

\medskip

\section{Introduction}
\label{sec:intro}

In this paper, we consider the Cauchy problem for the following quadratic nonlinear  Schr\"odinger system related to the Raman amplification in a plasma,
\begin{equation}\label{SNLS}\tag{NLS system}
\left\{
\begin{array}{ccc}
i\partial_t u + \Delta u +\overline{u}v &=& 0, \\
i\partial_t v +  \kk\Delta v+u^{2} &=& 0 \\
(u(0,x),v(0,x)) &=&(u_0,v_0),
\end{array}
\right.	
\end{equation}
where $u, v: \R \times \R^6 \rightarrow \C$, $u_0,v_0: \R^6 \rightarrow \C$, and $\kk>0$.
This system \eqref{SNLS} has applications in various physical problems, especially appears in the study of laser-plasma interactions; see \cite{CCC, CoDiSa, SCO} and references therein for a discussion of the physics of the problem.

In what follows, we use the vectorial notation $\ve{u}=(u_{1}, u_{2})$, and it is 
considered to be a column vector. The local well-posedness for the Cauchy problem of \eqref{SNLS} can be found in \cite{GaoMengXuZheng}.
Namely, if $\ve{u}_{0}\in \dot{H}^{1}(\R^{6})\times \dot{H}^{1}(\R^{6})$, then there exists a unique solution of \eqref{SNLS} defined on a maximal interval $I=(-T_{-}(u), T_{+}(u))$ and   satisfies conservation of energy and momentum:
\begin{align*}
E(\ve{u}(t)) &= \tfrac{1}{2} H(\ve{u}(t))  - \tfrac{1}{2} P(\ve{u}(t)) = E(\ve{u}_{0}), \\
\M(\ve{u}(t))&= \IM\int_{\R^6} \overline{u_{1}}\nabla u_{1}+\tfrac{1}{2}\overline{u_{2}}\nabla u_{2}\, dx=\M(\ve{u}_{0}),
\end{align*}
where
\begin{align*}
H(\ve{u}) &:= \|\nabla u_{1}\|^2_{L^2(\R^6)} + \tfrac{\kk}{2}\|\nabla u_{2}\|^2_{L^2(\R^6)}, & P(\ve{u}) &:= \RE\int_{\R^6} u_{1}^{2}\overline{u_{2}}  dx.
\end{align*}
In addition, Eq. \eqref{SNLS} enjoys several symmetries: If $\ve{u}=(u_{1}, u_{2})$ is a solution to \eqref{SNLS}, then 
\begin{enumerate}[label=\rm{(\roman*)}]
\item  by scaling: so is $\lambda^{-2}\ve{u}(\lambda^{-2}t,\lambda^{-1}x)$ with $\lambda>0$;
\item  by time translation invariance: so is $\ve{u}(t+t_{0},x)$ for $t_{0}\in\mathbb{R}$;
\item  by spatial translation invariance: so is $\ve{u}(t,x+x_{0})$ for $x_{0}\in\mathbb{R}^{6}$;
\item  by phase rotation invariance: so is $(e^{i\theta_{0}}u_{1}(t,x), e^{2i\theta_{0}}u_{2}(t,x))$ with $\theta_{0}\in\mathbb{R}$;
\item  by time reversal invariance: so is $\overline{\ve{u}(-t,x)}$.
\end{enumerate}

Let $Q$ be the radial, vanishing at infinity and positive solution of the nonlinear elliptic equation
\begin{align}\label{groundS}
	\Delta Q+Q^{2}=0.
\end{align}
We set $\Q:=(\sqrt{\kk}Q, Q)$. The ground state $\Q$ (it is unique up to translations and dilations, see \cite[Theorem 7.2]{NayaOzaTana}) play an important role in the study of the dynamics of the equation \eqref{SNLS}. In fact, we can write $ Q $ analytically
\begin{equation}\label{ground state-W}
Q(x) = \frac{1}{\left( 1 + \frac{\vert x \vert^2}{24} \right)^2} \in \dot{H}^1(\R^6).
\end{equation}
The study of nonlinear Schr\"{o}dinger systems with quadratic nonlinearities and their related models has garnered significant attention from physicists and mathematicians in recent years. Numerous studies focus on the well-posedness, orbital stability, and dynamic behavior of solutions with initial data below the ground state threshold. Three-wave interactions have been studied in various works; see, for example \cite{COCOOh, MESXu, MEXu, Pas2015, AA2018, GaoMengXuZheng}, while quadratic-type interactions are addressed in \cite{CoDiSa, DIFO2021, HaInuiNi, Ha2018, HaLiNa2011,InuiKisNI2019, NOPa2020, NPPA2021, WaYa2019, AeDiFo2021, AdemiNoguera2022}. 

In~\cite{GaoMengXuZheng}, the authors classified the behavior of radial solutions of \eqref{SNLS} satisfying $E(\ve{u}) < E(\Q)$ and established the following results.

\begin{theorem}[Sub-threshold scattering/blows up, \cite{GaoMengXuZheng, AdemiNoguera2022}]\label{Th1}
Fix $\kk>0$. Let $\ve{u}(t)$ be the corresponding solution to \eqref{SNLS} with radial initial data $\ve{u}_{0}\in \dot{H}^{1}(\R^{6})\times \dot{H}^{1}(\R^{6})$. 
\begin{enumerate}[label=\rm{(\roman*)}]
\item If $\ve{u}_{0}$ obeys $E(\ve{u}_0)<E(\Q)$ and  $H(\ve{u}_0)<H(\Q)$, then the solution $\ve{u}(t)$ exists globally and scatters at $t\to \pm\infty$.
\item If $\ve{u}_{0}$ obeys $E(\ve{u}_0)<E(\Q)$ and  $H(\ve{u}_0)>H(\Q)$, then the solution $\ve{u}(t)$ blows up in finite time for  
$\ve{u}_{0}$ radial with $|x|\ve{u}_{0}\in L^{2}(\R^{6})\times L^{2}(\R^{6})$ or 
$\ve{u}_{0}\in H^{1}(\R^{6})\times H^{1}(\R^{6})$.
\end{enumerate}
\end{theorem}

Our goal in this paper is to investigate the behavior of radial solutions of \eqref{SNLS} at the energy threshold $E(\ve{u}_0) = E(\Q)$.
We present our first result. In this result, we construct two special solutions that do not satisfy neither conclusion of Theorem~\ref{Th1}.

\begin{theorem}\label{Gcharc}
Fix $\kk>0$. There exist two solutions  $\G^{\pm}$ of \eqref{SNLS} with initial data $\G_0^{\pm}$ satisfying the following:
\begin{enumerate}[label=\rm{(\roman*)}]
\item $E(\G^{\pm})=E(\Q)$, $T_{+}(\G^{\pm})=\infty$ and 
\[
\lim_{t \to +\infty}\G^{\pm}(t)=\Q \quad\text{in $\dot{H}^{1}(\R^{6})\times \dot{H}^{1}(\R^{6})$;}
\]
\item $H(\G_{0}^{-})<H(\Q)$, $T_{-}(\G^{-})=\infty$ and $\G^{-}$ scatters in negative time.
\item $H(\G_{0}^{+})>H(\Q)$ and  $T_{-}(\G^{+})<\infty$.
\end{enumerate}
\end{theorem}

Using the solutions $\G^{\pm}$  from Theorem~\ref{Gcharc}, we can classify 
all possible behaviors of solutions with $E(\ve{u}_0) = E(\Q)$.

\begin{theorem}\label{TH22}
Fix $\kk>0$.
Let $\ve{u}(t)$ be the corresponding solution to \eqref{SNLS} with initial radial data 
$\ve{u}_{0}\in  \dot{H}^{1}(\R^{6})\times \dot{H}^{1}(\R^{6})$ and such that
\[
E(\ve{u}_{0})=E(\Q).
\]
\begin{enumerate}[label=\rm{(\roman*)}]
\item If $H(\ve{u}_{0})<H(\Q)$, 
then the solution $\ve{u}(t)$ to \eqref{SNLS} is global and, either $\ve{u}=\G^{-}$ up to the symmetries of the equation, or $\ve{u}(t)$scatters in both directions.
\item If  $H(\ve{u}_{0})=H(\Q)$, then $\ve{u}=\Q$ up to the symmetries of the equation.
\item If $H(\ve{u}_{0})>H(\Q)$, and $\ve{u}_{0}\in L^{2}(\R^{6})\times L^{2}(\R^{6})$
then ether $\ve{u}=\G^{+}$ or the solution $\ve{u}(t)$ to \eqref{SNLS} blows up in finite time.
\end{enumerate}
\end{theorem}

Theorem~\ref{TH22} is analogous to the threshold classification results found in studies on pure power NLS (see, \cite{DuyMerle2009, DuyckaertsRou2010}). Similarly, \cite{CAPA2022} classifies the dynamics at the mass-energy threshold for a cubic NLS system.  To the best of our knowledge, this is the first work that classifies the behavior of solutions at the energy threshold for an energy-critical NLS system.

By Theorem~\ref{TH22}, we obtain the following dynamical characterization of the ground state $\Q$.
\begin{corollary}
The ground state $\Q$ is the unique (up to the symmetries of the equation) radial solution with $ E(\ve{u}_{0}) = E(\Q) $ and $ H(\ve{u}_{0}) \leq H(\Q)$ that does not scatter in $\dot{H}^{1}(\R^{6})\times \dot{H}^{1}(\R^{6})$ for either positive nor negative times.
\end{corollary}

To obtain Theorems~\ref{Gcharc} and \ref{TH22}, we follow the strategy developed in \cite{DuyMerle2009}. Indeed, if we define the orbit $\G$ of the  ground state $\Q$,
\[
\G:= \left\{ \Q_{[\theta, \lambda]} :\,\, \theta \in \mathbb{R}, \lambda >0 \right\},
\]
where
\[
\Q_{[\theta, \lambda]}
:=\(\tfrac{\sqrt{\kappa}}{\lambda^{2}}e^{i\theta}{Q}\big(\tfrac{\cdot}{\lambda}\big), \tfrac{1}{\lambda^{2}}e^{2i\theta}{Q}\big(\tfrac{\cdot}{\lambda}\big)\)
\]
then we show that all solutions  $\ve{u}(t)$ of \eqref{SNLS} with initial data $\ve{u}_{0} \in \dot{H}^{1}(\mathbb{R}^{6}) \times \dot{H}^{1}(\mathbb{R}^{6})$ that lie on the threshold $E(\ve{u}) = E(\Q)$ exhibit one of the following six scenarios:
\begin{enumerate}
    \item Scattering for both $t \to \pm \infty$;
    \item Trapped by $\G$ for $t \to +\infty$ and scattering for $t \to -\infty$;
    \item Trapped by $\G$ for $t \to -\infty$ and scattering for $t \to +\infty$;
    \item Finite-time blow-up on both sides;
    \item Trapped by $\G$ for $t \to +\infty$ and finite-time blow-up for $t < 0$;
    \item Trapped by $\G$ for $t \to -\infty$ and finite-time blow-up for $t > 0$;
\end{enumerate}
Here, ``trapped by $\G$'' means that the solution remains within an $\mathcal{O}(\varepsilon)$ neighborhood of $\G$ with respect to the $\dot{H}^{1}(\mathbb{R}^{6}) \times \dot{H}^{1}(\mathbb{R}^{6})$ norm after some time (or before some time); see Propositions~\ref{CompacDeca} and \ref{SupercriQ} and Corollary~\ref{ClassC}.

Subsequently, we characterize the sets (2), (3), (5), and (6). However, we cannot directly apply the argument from \cite{DuyMerle2009} to obtain such a characterization. The main difficulty lies in the fact that, due to the asymmetry of the nonlinearity $(\overline{u}v, u^{2})$, the operator $J$ appears, multiplying the self-adjoint operator $\mathcal{L}$ in the linearized equation (see Section~\ref{S:Spectral}):
\begin{equation}\label{LIL}
\partial_{t}\ve{h} + J\L \ve{h} = \epsilon, \quad \text{where} \quad
\L := \begin{pmatrix}
0 & -L_{I} \\
L_{R} & 0
\end{pmatrix} \quad \text{and} \quad
J := \begin{pmatrix}
1 & 0 \\
0 & 2
\end{pmatrix},
\end{equation}
this operator $J$ prevents the application of the arguments developed in \cite{DuyMerle2009} to obtain the especial solutions to \eqref{SNLS} converging to the ground state $\Q$. To overcome this difficulty, we set the transformation
\[
T(\ve{u}) := T(u_{1}, u_{2}) = \left( \frac{u_{1}}{\sqrt{2}}, \frac{u_{2}}{{2}} \right).
\]
Note that if $\ve{u}$ is a solution to \eqref{SNLS}, then $T(\ve{u})$ is a solution to the system:
\begin{equation}\label{ANLS}
\left\{
\begin{array}{ccc}
i\partial_t u + \Delta u +2\overline{u}v &=& 0, \\
i\partial_t v +  \kk\Delta v+u^{2} &=& 0. \\
\end{array}
\right.	
\end{equation}
We show the existence and uniqueness of special solutions for the system \eqref{ANLS}; in this system, the operator $J$ does not appear in the linearized equation. Then, by applying $T^{-1}$, we can derive the existence and uniqueness of special solutions for the system \eqref{SNLS} and characterize the sets (2), (3), (5), and (6).

\subsection*{Outline}
In the rest of this introduction, we arrange the structure of the paper as follows. In Section \ref{S1:preli}, we revisit the Cauchy problem and discuss some variational characterizations of the ground state. Additionally, we recall the virial identities associated with \eqref{SNLS}.
In Section~\ref{S:Spectral}, we establish the spectral properties of the linearized operator  around the ground state $\mathcal{Q}$. In particular, we derive the quadratic form associated with the operator $\mathcal{L}$ and prove its coercivity. Using the coercivity of the quadratic form, we establish modulational stability in Section~\ref{S:Modula}. With modulational stability in hand, and following the argument developed in \cite[Section 3]{DuyMerle2009}, in Section~\ref{S: SubThres} we prove scattering and convergence to $\Q$ in the subcritical case (i.e., Proposition~\ref{CompacDeca}). Making use of the virial argument and following the approach as in \cite[Section 4]{DuyMerle2009}, in Section~\ref{S: SuperThres}, we prove  convergence to $\Q$ in the supercritical case (Proposition~\ref{SupercriQ}).  
In Sections~\ref{S:Existence} and \ref{S:uniq}, we establish the existence and uniqueness of special solutions for the system \eqref{ANLS}. In Section~\ref{S:proof}, we prove Theorem~\ref{TH22}. Finally, Appendices~\ref{S:A} and \ref{S:A2} are devoted to establishing the spectral properties of the real part and the imaginary part of the linearized operator.

\subsection*{Notation}
For any $s \geq 0$ we used $\dot{H}^{s}$ to denote $\dot{H}^{s}(\R^{6}:\C)\times \dot{H}^{s}(\R^{6}:\C)$.
We write ${H}^{s}$ to denote ${H}^{s}(\R^{6}:\C)\times {H}^{s}(\R^{6}:\C)$. Similarly, for each $p\geq1$ we write
$L^{p}$ to denote $L^{p}(\R^{6}:\C)\times L^{p}(\R^{6}:\C)$.

\section{Preliminaries}\label{S1:preli} 
In this section, we provide some results that will be useful throughout the paper.
First, we recall the well-posedness result for \eqref{SNLS}.

\subsection*{Local theory}

The following results can be found in \cite{GaoMengXuZheng}.
\begin{theorem}[see Theorem 2.5 in \cite{GaoMengXuZheng}]\label{LCP}
Fix $\ve{u}_{0}\in \dot{H}^{1}$. Then the following hold:
 \begin{enumerate}[label=\rm{(\roman*)}]

\item There exists $T_{+}(\ve{u}_{0})>0$, $T_{-}(\ve{u}_{0})>0$ and a unique solution 
$\ve{u}: (-T_{-}(\ve{u}_{0}), T_{+}(\ve{u}_{0}))\times \R^{6}\to \C$ to \eqref{SNLS} with initial data $\ve{u}(0)=\ve{u}_{0}$.
\item If $T_{+}=T_{+}(\ve{u}_{0})<+\infty$, then $\|\ve{u}\|_{\Lm^{4}_{t, x}((0, T_{+})\times \R^{6})}=+\infty$. Analogous statement holds in negative time.
\end{enumerate}
\end{theorem}

 The proof of the following result can be found in \cite[Remark 1.4]{GaoMengXuZheng}.

\begin{proposition}[Sufficient condition for scattering]
Any $\dot{H}^{1}$ global solution $u(t)$ in positive time ($T_{+}=+\infty$) that remains uniformly bounded in $L_{t,x}^{4}$, i.e.,
\[
\|u\|_{L^{4}_{t, x}([0, T_{+}) \times \R^{6})} < \infty,
\]
scatters in $\dot{H}^{1}$.
\end{proposition}

\begin{lemma}[Long time perturbation theory, see Proposition 2.7 in \cite{GaoMengXuZheng}]\label{ConSNLS}
Let $I\subset \R$ be a time interval containing $0$ and  let $\tilde{\ve{u}}$ be  solution to \eqref{SNLS} on $I$. 
Suppose
\[
\sup_{t\in I}\| \tilde{\ve{u}} (t) \|_{H^{1}}\leq L\qtq{and} 
	\| \tilde{\ve{u}}  \|_{L_{t,x}^{4}(I\times\R^{6})}\leq L
\]
for some $L>0$. Also,  assume  the smallness conditions 
\[
	\|\ve{u}_{0}-\tilde{\ve{u}}_{0} \|_{\dot{H}^{1}}\leq \epsilon
\]
for some $0<\epsilon<\epsilon_{0}=\epsilon_{0}(\mbox{L})>0$. 
Then there exists a unique solution $\ve{u}$ to  \eqref{SNLS} with initial data $\ve{u}_{0}$ satisfying
\[
	\sup_{t\in I}\|\ve{u}(t)-\tilde{\ve{u}}(t)\|_{\dot{H}^{1}}\leq C(L)\epsilon
	\qtq{and}	\|\ve{u}\|_{L_{t,x}^{4}(I\times\R^{6})}\leq C(L).
\]
\end{lemma}

\subsection*{Varational analysis}

We review some results on the ground states $\Q$ for \eqref{SNLS}, along with some variational properties.
\begin{proposition}[see \cite{NayaOzaTana}]\label{TheGN}
Let $\kk>0$. Then the infimum 
\begin{equation}\label{GNI}
J_{\text{min}}:=\inf_{ f\in \B}\frac{[H(\ve{f})]^{3}}{[P(\ve{f})]^{2}},
\end{equation}
where $\B:=\left\{\ve{f}\in \dot{H}^{1}\cap L^{3}: P(\ve{f})>0 \right\}$ is  attained if and only if 
\[\ve{f}(x)=(e^{i\theta}m\sqrt{\kk} Q(n (x-x_{0})), e^{2i\theta}m Q(n (x-x_{0})))\]
for some $m>0$, $n>0$, $\theta\in \R$ and $x_{0}\in \R^{6}$.
\end{proposition}

As a consequence of Proposition~\ref{TheGN}, we have the following Gagliardo-Nirenberg-type inequality
\begin{align}\label{GNIS}
	P(\ve{f})\leq C_{GN}[H(\ve{f})]^{\frac{3}{2}},
\end{align}
where $C_{GN}=J_{\text{min}}^{-\frac{1}{2}}$. Since $\Q$ is a optimizer of \eqref{GNI} (recall that $\Q=(\sqrt{\kk}Q, Q)$), we have
$C_{GN}=\tfrac{P(\Q)}{[H(\Q)]^{\frac{3}{2}}}$. We also have the following Pohozaev's identity 
\[
H(\Q)=\tfrac{3}{2}P(\Q).
\]

Moreover, by \cite[Lemma 3.6]{GaoMengXuZheng}, for any $ \ve{v} \in \dot{H}^1$,
\begin{align}\label{DV}
	 |P(v)| \leq C(\kappa)[H(v)]^{\frac{3}{2}}, \qtq{ where $C(\kappa) = \sqrt{\frac{8}{27\kappa}}$}.
\end{align}

\begin{lemma}\label{ConvexIn}
If $\ve{f}\in \dot{H}^{1}$ and $H(\ve{f})\leq H(\Q)$, then
\[
H(\ve{f})E(\Q)\leq H(\Q)E(\ve{f}).
\]
\end{lemma}
\begin{proof}
It follows from Gagliardo-Nirenberg's inequality \eqref{GNIS} that
\begin{align}\label{AUXGN}
	E(\ve{f})\geq \tfrac{1}{2}H(\ve{f})-\tfrac{1}{2}C_{GN}[H(\ve{f})]^{\frac{3}{2}}=\Phi(H(\ve{f})),
\end{align}
where 
\[\Phi(y)=\tfrac{1}{2}y-\tfrac{1}{2}C_{GN}y^{\frac{3}{2}}
\]
for $y\geq 0$. An elementary calculations shows that $\Phi(H(\Q))=E(\Q)$. Since $\Phi$ is concave on $(0, +\infty)$
and $\Phi(0)=0$, we find
\[
\Phi(\alpha H(\Q))\geq \alpha\phi(H(\Q))=\alpha E(\Q)
\qtq{for all $\alpha\in(0,1)$}.
\]
Choosing $\alpha=\frac{H(\ve{f})}{H(\Q)}$, then \eqref{AUXGN} implies the lemma.
\end{proof}

For a function $\ve{u}$ defined on $\R^{6}$, we define
\[
\ve{u}_{[x_{0}, \theta_{0}, \lambda_{0}]}
= \left( \tfrac{1}{\lambda^{2}_{0}} e^{i\theta_{0}} u_{1}\left(\tfrac{x - x_{0}}{\lambda_{0}}\right), \tfrac{1}{\lambda^{2}_{0}} e^{2i\theta_{0}} u_{2}\left(\tfrac{x - x_{0}}{\lambda_{0}}\right) \right),
\]
where $\ve{u} = (u_{1}, u_{2})$. Using the variational characterization of $ \Q $ given in Proposition~\ref{TheGN}, inequality \eqref{DV} and profile decomposition in $\dot{{H}}^{1}$, one can show that:

\begin{proposition}\label{profiQ}
Let $\kk>0$. Then there exists a function $\epsilon=\epsilon(\rho)$ such that for all $\ve{u}\in \dot{H}^{1}$
satisfying $E(\ve{u})=E(\Q)$ we have
\[
\inf_{y\in \R^{6}, \theta\in \R, \lambda>0}\|\ve{u}_{[y, \theta, \lambda]} -\Q\|_{\dot{{H}}^{1}}
\leq \epsilon(\delta(\ve{u}))\qtq{with} \lim_{\rho\to 0}\epsilon(\rho)=0,
\]
where $\delta(\ve{u}):=|H(\ve{u})-H(\Q)|$.
\end{proposition}
\begin{proof}
With the appropriate modifications, the proof follows the same lines as in \cite[Proposition 2.7]{MIWUGU2015}.
\end{proof}

\subsection{Virial identities}
For $R>1$ we define the function
\[
w_{R}(x)=R^{2}\phi\(\tfrac{x}{R}\)
\quad \text{and}\quad w_{\infty}(x)=|x|^{2},
\]
where $\phi$ is a real-valued and radial function so that
\[
\phi(x)=
\begin{cases}
|x|^{2},& \quad |x|\leq 1\\
0,& \quad |x|\geq 2,
\end{cases}
\quad \text{with}\quad 
|\partial^{\alpha}\phi(x)|\lesssim |x|^{2-|\alpha|}.
\]
Consider the localized virial functional ($\ve{u}=(u,v)$)
\[
I_{R}[\ve{u}]=2\kk\IM\int_{\R^{6}} \nabla w_{R}(x) \cdot 
(2 \nabla u(t,x) \overline{u(t,x)}+ \nabla {v}(t,x) \overline{v(t,x)}) dx.
\]

We need the following result; cf. \cite[Lemma 3.9]{GaoMengXuZheng}.

\begin{lemma}\label{VirialIden}
Let $R\in [1, \infty]$. Suppose $\ve{u}(t)$ solves \eqref{SNLS}. Then
\begin{equation}\label{LocalVirial}
\frac{d}{d t}I_{R}[\ve{u}]=F_{R}[\ve{u}(t)],
\end{equation}
where
\begin{align*}
F_{R}[\ve{u}]&:=\int_{\R^{6}}(-2\kk \Delta \Delta w_{R})[|u|^{2}+\tfrac{\kk}{2}|v|^{2}]-
2\kk\Delta[w_{R}(x)]\overline{v}u^{2}\,dx\\
&+8\kk\RE\int_{\R^{6}} [\overline{u_{j}} u_{k} +\tfrac{\kk}{2}\overline{v_{j}} v_{k}]\partial_{jk}[w_{R}]dx.
\end{align*}
In particular, when $R=\infty$ we have $F_{\infty}[u]=8\kk[2H(\ve{u})-3P(\ve{u})]$.
\end{lemma}

Note that given the specifications of the weight $w_{R}$ defined above, we have that (with $\phi(r)=\phi(|x|)$)
\begin{align*}
\RE\int_{\R^{6}} \Big[ \overline{u_{j}} u_{k} +\tfrac{\kk}{2}\overline{v_{j}} v_{k} \Big] \partial_{jk}[w_{R}] \, dx = \int_{\R^{6}} \Big[ |\nabla u|^{2} +\tfrac{\kk}{2}|\nabla v|^{2} \Big] \partial^{2}_{r}w_{R} \, dx.
\end{align*}

We abuse the notation $ \ve{u}_{[\theta_{0}, \lambda_{0}]} $ to denote $ \ve{u}_{[x_{0}, \theta_{0}, \lambda_{0}]} $ for $ x_0 \equiv \ve{0} \in \R^6 $. As a direct consequence of the previous lemma, we obtain the following results.
\begin{lemma}\label{Virialzero}
Let $R\in [1, \infty]$, $\theta\in \R$ and $\lambda>0$. Then
\[
I_{R}[\Q_{[\theta, \lambda]}]=0.
\]
\end{lemma}

\begin{lemma}\label{VirialModulate}
Let $\ve{u}$ be the solution of \eqref{SNLS} on an interval $I$. Let $R\in [1, \infty]$, $\chi: I\to \R$, $\theta: I\to \R$,
$\lambda: I\to \R^{\ast}$. Then for all $t\in I$,
\begin{align}\nonumber
	\frac{d}{d t}I_{R}[\ve{u}]&=F_{\infty}[\ve{u}(t)]\\ \label{Modu11}
	                     &+F_{R}[\ve{u}(t)]-F_{\infty}[\ve{u}(t)]\\\label{Modu22}
											 &-\chi(t)\big\{F_{R}[\Q_{[\theta(t), \lambda(t)]}]-F_{\infty}[\Q_{[\theta(t), \lambda(t)]}]\big\}.
\end{align}
\end{lemma}

\section{Spectral properties of the linearized operator}\label{S:Spectral}

We establish some spectral properties of the linearized operator that will be useful in the following sections. Let $\ve{u}(t)$ be a solution to \eqref{SNLS} and define 
\[
\ve{h}(t,x):=\ve{u}(t,x)-\Q(x).
\]
Then, defining the operators $L_{R}$ and $L_{I}$ (acting on $\dot{H}^{1}(\R^6;\R)$) via
\begin{align*}
L_{R}&=
\begin{bmatrix}
-\Delta & 0 \\
0 & -\tfrac{\kk}{2}\Delta 
\end{bmatrix}
-
\begin{bmatrix}
Q & \sqrt{\kk}Q \\
\sqrt{\kk}Q & 0 
\end{bmatrix},
\\
L_{I}&=
\begin{bmatrix}
-\Delta & 0 \\
0 & -\tfrac{\kk}{2}\Delta 
\end{bmatrix}
+
\begin{bmatrix}
Q & -\sqrt{\kk}Q \\
-\sqrt{\kk}Q & 0 
\end{bmatrix}
\end{align*}
and writing
\begin{equation}\label{Resi}
\begin{split}
R(\ve{h})=R(h,g)&=(\overline{h}g, h^{2}),\qtq{where $\ve{h}=(h, g)$,}
\end{split}
\end{equation}
we see that $\ve{h}$ satisfies the equation
\begin{equation}\label{Decomh}
\partial_{t}\ve{h}+J\L \ve{h}=iR\ve{h}, \quad \text{where} \quad
\L:=\begin{pmatrix}
0 & -L_{I} \\
L_{R} & 0
\end{pmatrix},
\end{equation}
and \[
J(\ve{u})=J(u_{1}, u_{2})=(u_{1}, 2u_{2}).
\]
Notice that for $\ve{h}=(h,g)$ with $h=h_{1}+ih_{2}$ and $g=g_{1}+ig_{2}$, we have
\[
\L\ve{h}=\L (h,g)=i\left[L_{R}\begin{pmatrix}
           h_{1} \\
           g_{1} 
  \end{pmatrix}    
+iL_{I}\begin{pmatrix}
           h_{2} \\
           g_{2} 
  \end{pmatrix}\right]  
=
i(-\Delta h-\overline{h}Q-\sqrt{\kk}Qg, -\tfrac{\kk}{2}\Delta g-\sqrt{\kk}Qh).
\]

Moreover, we can write \eqref{Decomh} as a Schr\"odinger equation (recall that $\ve{h}=(h, g)$)
\[
(i\partial_{t}h, i\partial_{t}g)+(\Delta h, {\kk}\Delta g)+K(\ve{h})=-R(\ve{h}),
\]
where $K(\ve{h})=K(h,g)=(\overline{h}Q+\sqrt{\kk}Qg, 2\sqrt{\kk}Qh)$.

Consider the orthogonal directions 
\[
\Q=(\sqrt{\kk}Q, Q)\quad i\Q_{1}=i(\sqrt{\kk}Q,  2Q),\quad \partial_{j} \Q=(\sqrt{\kk}\partial_{j} Q, \partial_{j} Q)
 \]
for $j=1,\ldots,6$, and
\[
\Lambda \Q=(\sqrt{\kk}\Lambda Q,  \Lambda Q) \qtq{where} \Lambda Q=2Q+x\cdot\nabla Q
\]
in the Hilbert space $\dot{H}^{1}=\dot{H}^{1}(\R^{6}; \C)\times \dot{H}^{1}(\R^{6}; \C)$.

Our next step is to study the spectrum of the operators $ L_{R} $ and $ L_{I}$. 

By the Weyl theorem, the spectra of $L_{R}$ and $L_{I}$ consist of essential spectrum in $[0, \infty)$. From Remark~\ref{KerLI} below, we see that $L_{I}$ is nonnegative and
${\rm Ker} \left\{L_{I}\right\}=\mbox{span}\left\{\Q_{1} \right\}$. On the other hand, $L_{R}$ has only one negative eigenvalue. Moreover, the second eigenvalue is $0$
and (cf. Remark~\ref{KerLR})
\[
{\rm Ker} \left\{L_{R}\right\}=\mbox{span}\left\{\Lambda\Q, \partial_{j}Q: j=1,2, 3 \right\}.
\]
By a direct calculation
\[
L_{R}(\partial_{j} \Q)
	=\ve{0},
\quad
L_{R}(\Lambda\Q)
	=\ve{0}
\]
and
\[
L_{I}\Q_{1}
	=\ve{0}.
\]
In particular,  we see that
\[
\L \partial_{j}\Q
	=
\L \Lambda \Q
=
\L i\Q_{1}
	=\ve{0}.
	\]

We denote by $\F(\ve{u}, \ve{v})$ the bilinear symmetric form 
\begin{equation}\label{Quadratic}
\F(\ve{u}, \ve{v}):=\tfrac{1}{2}\<L_{R} \RE\ve{u}, \RE\ve{v}\>
	+\tfrac{1}{2}\<L_{I}\IM\ve{u}, \IM\ve{v}\>.
\end{equation}
 Also, we denote $\F(\ve{u},\ve{u})$ by $\F(\ve{u})$.

\begin{lemma}\label{Coer11}
Fix $\kk>0$. There exists $C>0$ such that for every $\ve{v}\in \dot{H}^{1}(\R^{6}: \R)\times \dot{H}^{1}(\R^{6}: \R)$ satisfying
\begin{align}\label{SOrt11}
	\F(\ve{v}, \Q )=(\ve{v}, \Lambda \Q )_{\dot{H}^{1}}=(\ve{v},  \partial_{j}\Q)_{\dot{H}^{1}}=0
\end{align}
for $1\leq j\leq 6$, then we have
\begin{align}\label{Fine}
	\<L_{R}\ve{v},\ve{v}
	\>
	\geq C \|\ve{v}\|^{2}_{\dot{H}^{1}}.
\end{align}
\end{lemma}
\begin{proof}
Fix $\kk>0$. First we show that if $\ve{v}$ satisfies \eqref{SOrt11}, then $\F(\ve{v}):=\tfrac{1}{2}\<L_{R}\ve{v},\ve{v}\>>0$.
Indeed, if not, there exists $\ve{g}\in \dot{H}^{1}\setminus\left\{0\right\}$ such that
\begin{align}\label{Ort22}
	\F(\ve{g}, \Q )=(\ve{g}, \Lambda \Q )_{\dot{H}^{1}}=(\ve{g},  \partial_{j}\Q)_{\dot{H}^{1}}=0,
\end{align}
and $\F(\ve{g})\leq 0$. Consider $E=\text{span}\left\{\ve{g}, \Q, \Lambda \Q, \partial_{j}\Q: j=1,2,\ldots,6 \right\}$.
Since $\F(\Lambda \Q, \ve{h})=\F( \partial_{j}\Q, \ve{h})=0$ for all $h\in \dot{H}^{1}$, $\F(\ve{g}, \Q )=0$ and $\F(\Q)<0$ (see Remark~\ref{PLf}), we infer that
$\F(\ve{v})\leq 0$ for all $\ve{v}\in E$.

On the other hand, notice that $E$ is a subspace of dimension $9$. 
Assume that
\[
\sum^{6}_{j=1}\alpha_{j}\partial_{j}\Q+\beta\Lambda \Q+\gamma \Q+\mu \ve{g}=0.
\]
Since $\partial_{j}\Q$, $\Lambda \Q$, and $\ve{g}$
are orthogonal in $\dot{H}^{1}$ we infer (recall $(\Lambda \Q, \Q)_{\dot{H}^{1}}=(\partial_{j}\Q, \Q)_{\dot{H}^{1}}=0$) that 
$\alpha_{j}=0$ and $\beta=0$. Therefore, we get
\[
\gamma \Q+\mu \ve{g}=0.
\]
Now since $\F(\ve{g}, \Q )=0$ and $\F(\Q)<0$ we see that $\gamma=0$. Finally, as  $\ve{g}\in \dot{H}^{1}\setminus\left\{0\right\}$,
we obtain that $\mu=0$. By Lemma~\ref{Apx11} we know that $\F$ is definite positive on a subspace of co-dimension $8$, yielding a contradiction.

This proves that if $\ve{v}$ satisfies \eqref{SOrt11}, then $\<L_{R}\ve{v},\ve{v}\>>0$. As the quadratic form $\F(\cdot)$ is a compact perturbation of $\|\cdot\|^{2}_{\dot{H}^{2}}+\tfrac{k}{2}\|\cdot\|^{2}_{\dot{H}^{1}}$, we obtain \eqref{Fine}.
\end{proof}

\begin{lemma}\label{CoerLi11}
Fix $\kk>0$. There exists $C>0$ such that for every $\ve{v}\in \dot{H}^{1}(\R^{6}: \R)\times \dot{H}^{1}(\R^{6}: \R)$ satisfying
\begin{align}\label{SOrt1111}
(\ve{v},  \Q_{1})_{\dot{H}^{1}}=0,
\end{align}
 then we have
\begin{align}\label{Fine22}
	\<L_{I}\ve{v},\ve{v}
	\>
	\geq C \|\ve{v}\|^{2}_{\dot{H}^{1}}.
\end{align}
\end{lemma}
\begin{proof}
As in the Lemma~\ref{Coer11}, it suffices to show that for all $\ve{v}\in \dot{H}^{1}$ satisfying \eqref{SOrt1111} we have 
$\F_{-}(\ve{v}):=\tfrac{1}{2}\<L_{I}\ve{v},\ve{v}\>>0$. Assume by 
contradiction that there exists $\ve{g}\in \dot{H}^{1}\setminus\left\{0\right\}$ such that
\begin{align}\label{contra1}
	(\ve{g}, \Q_{1})_{\dot{H}^{1}}=0,
\end{align}
and $\F_{-}(\ve{g})\leq 0$. Note that $L_I(\Q_{1})=0$. Since $\F_{-}(\Q_{1}, \ve{h})=0$ for all $\ve{h}\in \dot{H}^{1}$, we see that
$\F_{-}(\ve{v})\leq 0$ for $\ve{v}\in E$, where $E=\text{span}\left\{{\ve{g}, \Q_{1}}\right\}$. Moreover, by \eqref{contra1} we have that
$E$ is a subspace of dimension $2$, which is a contradiction with Lemma~\ref{Apx22}. 
\end{proof}

\begin{remark}\label{KerLI}
Since $L_I \Q_{1}=0$, Lemma~\ref{CoerLi11} implies $L_{I}\geq 0$ and $\text{Ker}(L_I)=\text{span}\left\{\Q_{1}\right\}$. 
\end{remark}


 Combining Lemmas~\ref{Coer11} and \ref{CoerLi11}, we obtain the following proposition.

\begin{proposition}\label{CoerQuadra}
Fix $\kk>0$. There exists $C>0$ such that for every $\ve{h}\in G^{\bot}$, we have
\[
\F(\ve{h})\geq C \|\ve{h}\|^{2}_{\dot{H}^{1}},
\]
where
\[
G^{\bot}:=\left\{\ve{h}\in \dot{H}^{1},  \F(\Q, \ve{h})=(i\Q_{1},\ve{h})_{\dot{H}^{1}}=
 (\Lambda \Q, \ve{h})_{\dot{H}^{1}}=
(\partial_{j} \Q, \ve{h})_{\dot{H}^{1}}=0: j=1,\ldots,6\right\}.
\]
\end{proposition}

We record the following identity, which follows from direct computation:
\begin{lemma}\label{Fuu}
Let $\ve{h}=(u,v)\in \dot{H}^{1}$ and assume $E(\Q+\ve{h})=E(\Q)$.
Then we have
\begin{align*}
\F(\ve{h})&=\tfrac{1}{2}\int_{\R^{6}}\overline{v}u^{2}.
\end{align*}
In particular, $|\F(\ve{h})|\lesssim \|\ve{h}\|^{3}_{\dot{H}^{1}}$.
\end{lemma}

\section{Modulation analysis}\label{S:Modula}

Consider a solution $\ve{u}$ to \eqref{SNLS} with initial data $\ve{u}_{0}$ satisfying
\[
E(\ve{u})=E(\Q).
\]
We denote 
\[
\delta(t):=\left|H(\ve{u}(t))-H(\Q)\right|.
\]

Let $\delta_{0}>0$ be a small parameter and define the open set
\[
I_{0}=\left\{t\in [0, \infty):\delta(t)<\delta_{0}\right\}.
\] 
We will prove the following proposition. 

\begin{proposition}\label{ModilationFree}
Fix $\kk>0$. For $\delta_{0}>0$ sufficiently small, there exist functions $\theta: I_{0}\to \R$, $\lambda: I_{0}\to \R$
$\alpha: I_{0}\to \R$, $y: I_{0}\to \R^{6}$, and $\ve{h}:I_0\to \dot{H}^1$ such that we can write
\begin{equation}\label{DecomUFree}
\ve{u}_{[y(t), \theta(t), \lambda(t)]}(t)=(1+\alpha(t))\Q+\ve{h}(t)\qtq{for all}t\in I_0,
\end{equation}
with
\begin{equation}\label{EstimateOne}
|\alpha(t)|\sim \|\ve{h}(t)\|_{\dot{H}^{1}}\sim\delta(t)
\end{equation}
and
\begin{equation}\label{EstimateFree}
|y^{\prime}(t)|+|\theta^{\prime}(t)|+|\alpha^{\prime}(t)|+\frac{|\lambda^{\prime}(t)|}{|\lambda(t)|}
\lesssim\lambda^{2}(t)\delta(t).
\end{equation}
\end{proposition}

Using the variational characterization of $\Q$ (cf. Lemma~\ref{profiQ}) and the Implicit Function Theorem, one can obtain the following orthogonal decomposition.  We will omit the proof, as it is essentially the same given in \cite{DuyMerle2009}.

\begin{lemma}\label{ExistenceFree}
Fix $\kk>0$. There exist $\delta_{0}>0$, a positive function $\epsilon(\delta)$ defined for $0<\delta\leq \delta_{0}$ and functions $\theta: I_{0}\to \R$, $\lambda:I_{0}\to \R$ and $y: I_{0}\to \R^{6}$ such that if
$\delta(t)<\delta_{0}$, then
\begin{equation}\label{TaylorF}
\|\ve{u}_{[y, \theta, \lambda]} -\Q\|_{\dot{H}^{1}}\leq \epsilon(\delta).
\end{equation}
Moreover, the mapping $\ve{u}\mapsto (y, \theta, \lambda)$ is $C^{1}$ and $\epsilon(\delta)\to 0$ as $\delta\to 0$. Finally, the functions $\theta(\cdot)$, $\lambda(\cdot)$ and $y(\cdot)$ are chosen to impose the following orthogonality condition:
\begin{equation}\label{OrtFree}
\ve{u}_{[y, \theta, \lambda]}\bot \text{span}\left\{\nabla \Q, i\Q_{1}, \Lambda\Q  \right\}.
\end{equation}
\end{lemma}

With $(y(t), \theta(t), \lambda(t))$ as in Lemma~\ref{ExistenceFree}, we write
\begin{equation}\label{DefH}
\ve{u}_{[y(t), \theta(t), \lambda(t)]}(t)=(1+\alpha(t))\Q+\ve{h}(t)\qtq{for all}t\in I_0,
\end{equation}
where
\[
\alpha(t)+1=\tfrac{1}{\F(\Q,\Q)}\F(\Q, \ve{u}_{[y(t), \theta(t), \lambda(t)]}).
\]
Thus, we have the following orthogonally conditions for $\ve{h}$: 
\begin{equation}\label{h: Cond}
\ve{h}\bot \text{span}\left\{\nabla \Q, i\Q_{1}, \Lambda\Q  \right\}
\qtq{and $\F(\Q,\ve{h})=0$.}
\end{equation}
We note that $L_{R}(\Q)=(\sqrt{\kk}\Delta Q, \frac{\kk}{2}\Delta Q)$. In particular, by \eqref{h: Cond} we see that 
\[
\left(\big(\sqrt{\kk} Q, \tfrac{\kk}{2} Q\big), \ve{h}\right)_{\dot{H}^{1}}=0.
\]

The following lemma relates the parameters $\delta(t)$, $\lambda(t)$ $\alpha(t)$, and $\ve{h}(t)$. 
\begin{lemma}\label{BoundI}
Fix $\kk>0$. For all $t\in I_0$, we have
\begin{equation}\label{DeltaBound}
\delta(t)\sim |\alpha(t)|\sim \|\ve{h}(t)\|_{\dot{H}^{1}}.
\end{equation}
\end{lemma}
\begin{proof}
Writing $\ve{v}=\ve{u}_{[y(t), \theta(t), \lambda(t)]}(t)-\Q=\ve{h}+\alpha(t)\Q$ we see that
\begin{align}\label{EquV}
	H(\ve{v})=\alpha^{2}H(\Q)+H(\ve{h}),
\end{align}
where we have used that $\left(\big(\sqrt{\kk} Q, \tfrac{\kk}{2} Q\big), \ve{h}\right)_{\dot{H}^{1}}
=2\F(\Q, \ve{h})=0$ (cf. \eqref{h: Cond}).  Notice that $H(\ve{v})\sim \|\ve{v}\|^{2}_{\dot{H}^{1}}$ is small when $\delta(t)$ is small.

Notice also that (recall that $\F(\Q)<0$)
\[
\F(\ve{v})=\F(\ve{h})+\alpha^{2}\F(\Q)=\F(\ve{h})-\alpha^{2}|\F(\Q)|.
\]
Thus, by Lemma~\ref{Fuu} we infer that $|\F(\ve{h})-\alpha^{2}|\F(\Q)||\leq C\|\ve{v}\|^{3}_{\dot{H}^{1}}$. Moreover, by 
Proposition~\ref{CoerQuadra} (recalling the orthogonally conditions for $\ve{h}$, \eqref{h: Cond})
 we infer that  $\|\ve{h}\|^{2}_{{\dot{H}}^{1}}\sim \F(\ve{h})$. Therefore,
\begin{align}\label{TwoI}
	\alpha^{2}\leq O(\|\ve{h}\|^{2}_{{\dot{H}}^{1}}+\|\ve{v}\|^{3}_{{\dot{H}}^{1}})
\qtq{and} 
\|\ve{h}\|^{2}_{{\dot{H}}^{1}}\leq O(\alpha^{2}+\|\ve{v}\|^{3}_{{\dot{H}}^{1}}).
\end{align}
Since $H(\ve{v}) \sim \|\ve{v}\|^{2}_{\dot{H}^{1}}$, then combining \eqref{EquV} and \eqref{TwoI} we get for $\delta_{0}$ small
\[
|\alpha|\sim \|\ve{h}\|_{{\dot{H}}^{1}} \sim \|\ve{v}\|_{{\dot{H}}^{1}}.
\]
Finally, as
\[
\delta(t)=|H(\ve{v})-H(\Q)|=|H(\ve{v})-\alpha|\F(\Q)||,
\]
we get $\delta \sim |\alpha|$. This proves the result.
\end{proof}

\begin{lemma}\label{BoundII}
 Let $(y(t), \theta(t), \lambda(t))$ be as in Lemma~\ref{ExistenceFree}  and $h(t)$ and $\alpha(t)$ be as in \eqref{DefH}. Then, taking a smaller $\delta_{0}$ if necessary, we have
\begin{align}
\label{alfaE}
|y^{\prime}(t)|+|\theta^{\prime}(t)|+|\alpha^{\prime}(t)|+\frac{|\lambda^{\prime}(t)|}{|\lambda(t)|}
\lesssim\lambda^{2}(t)\delta(t).
\end{align}
\end{lemma}
\begin{proof}
With Lemma~\ref{BoundI} in hand, the proof is very similar to that of \cite[Lemma~3.7]{DuyMerle2009}. 
Indeed, consider the variables $y$ and $s$ given by
\[
\lambda(t)z=x \qtq{and} ds=\lambda^{2}(t)dt.
\] 
Then the equation \eqref{SNLS} and Lemma~\ref{BoundI} yields
\[
i\partial_{s} \ve{h}-i\alpha_{s}\Q+\theta_{s}\Q_{1}+i\tfrac{|\lambda_{s}|}{|\lambda|}\Lambda \Q-iy_{s}\cdot \nabla \Q
+\tilde{J}\ve{h}=O(\epsilon(s)),
\]
where $\tilde{J}(f,g)=(f, \kk g)$ and 
\[\epsilon(s)=:\delta(u(t(s)))[\delta(u(t(s)))+|y_{s}|+|\theta_{s}|+|\alpha_{s}|
+\tfrac{|\lambda_{s}|}{|\lambda|}].\]
Using \eqref{h: Cond}, we obtain $\partial_{t}\ve{h}\bot \text{span}\left\{\nabla \Q, i\Q_{1}, \Lambda\Q  \right\}$
and $\F(\Q,\partial_{t}\ve{h})=0$. Therefore, integrating by parts we infer that
\[
|y_{s}|+|\theta_{s}|+|\alpha_{s}|
+\tfrac{|\lambda_{s}|}{|\lambda|}=0(\epsilon(s)+\delta(u(t(s)))).
\]
If $\delta_{0}$ is chosen small enough we obtain the result.
\end{proof}

\begin{proof}[Proof of Proposition~\ref{ModilationFree}]
The proof of Proposition~\ref{ModilationFree} is a direct consequence of Lemmas~\ref{BoundI} and~\ref{BoundII}.
\end{proof}



\section{Convergence for subcritical threshold solution}\label{S: SubThres}

The main objective of this section is to establish the following result.
\begin{proposition}\label{CompacDeca}
Let $\ve{u}$ be a radial solution of \eqref{SNLS} on $I=(T_{-}, T_{+})$ satisfying 
\begin{align}\label{PropCon11}
E(\ve{u}_{0})=E(\Q) \qtq{and} H(\ve{u}_{0})<H(\Q).
\end{align}
Then $I=\R$. Moreover, if 
\begin{align}\label{ScaTa}
\|\ve{u}\|_{\Lm^{4}_{t, x}((0, \infty)\times \R^{6})}=\infty,
\end{align}
then there exist $\theta\in \R$, $\lambda>0$, $c>0$ and $C>0$ such that
\[
\|\ve{u}(t)-\Q_{[\theta, \lambda]}\|_{\dot{H}^{1}}\leq Ce^{-c t}
\qtq{for all $t\geq 0$.}
\]
Similar result holds for negative time.
\end{proposition}

As a corollary, we will obtain the following: 
\begin{corollary}\label{ClassC}
There is no radial solution to equation \eqref{SNLS} satisfying \eqref{PropCon11} and 
\begin{equation}\label{Infinity10}
\|u\|_{\Lm^{4}_{t, x}((0, \infty)\times \R^{6})}=
\|u\|_{\Lm^{4}_{t, x}((-\infty,0)\times \R^{6})}=
\infty.
\end{equation}
\end{corollary}

Before proving Proposition~\ref{PropCon11}, we establish some lemmas. Note that from \eqref{PropCon11}, the variational characterization of $\Q$ in Proposition~\ref{TheGN}, and a continuity argument, we have
\[
H(\ve{u}(t)) < H(\Q) \quad \text{for all } t \text{ in the domain of existence.}
\]
Combining this with Theorem~\ref{Th1}~(i), the concentration-compactness principle, and stability theory (Lemma~\ref{ConSNLS}), as in \cite[Section 4.3]{GaoMengXuZheng}, we can show:

\begin{lemma}[Compactness]\label{Compacness11}
Suppose that $\ve{u}(t)$ is a radial solution of \eqref{SNLS} of maximal interval of existence $I=[0, T_{+})$ such that  satisfies \eqref{PropCon11} and 
\begin{align}\label{NscaF}
	\|\ve{u}\|_{\Lm^{4}_{t, x}((0, T_{+})\times \R^{6})}=\infty.
\end{align}
 Then there exists a continuous function 
$\lambda:[0, T_{+}) \to \R$ such that 
 \begin{equation}\label{CompactX}
\left\{\ve{u}_{[\lambda(t)]}: t\in [0, T_{+})\right\} \quad
\text{is pre-compact in $\dot{H}^{1}(\R^{6})$}.
\end{equation}
\end{lemma}

The compactness in Lemma \ref{Compacness11} is the last piece of the puzzle to show the existence of global solutions when
\begin{equation}\label{leqleq}
E(\ve{u}_{0}) \leq E(\Q) \qtq{and} H(\ve{u}_{0}) \leq H(\Q).
\end{equation}

\begin{lemma}[Global solution]\label{GlobalW}
Suppose that $\ve{u}(t)$ is a radial solution of \eqref{SNLS} with maximal interval of existence $I=[T_{-}, T_{+})$.  If 
\eqref{leqleq} holds, then we have $I=\R$.
\end{lemma}
\begin{proof}
If $H(\ve{u}_{0})=H(\Q)$, Lemma~\ref{ConvexIn} implies that $E(\ve{u})=E(\Q)$. Then, by Proposition~\ref{TheGN} we find 
$\ve{u}_{0}=\Q_{[\theta_{0}, \lambda_{0}]}$, which implies that the solution $\ve{u}$ is global.

On the other hand, suppose that $H(\ve{u}_{0})<H(\Q)$ and $E(\ve{u})<E(\Q)$. By Theorem~\ref{Th1}, the solution 
$\ve{u}$ is global and scatters.

Finally, suppose that $H(\ve{u}_{0})<H(\Q)$ and $E(\ve{u})=E(\Q)$. If $\|\ve{u}\|_{\Lm^{4}_{t, x}(I\times \R^{6})}<\infty$,
then by the finite blow up criterion, we see that $\ve{u}$ is a global solution.

Next, if $\|\ve{u}\|_{\Lm^{4}_{t, x}([0, T_{+})\times \R^{6})}=\infty$, Lemma~\ref{Compacness11} implies that there exists a continuous function
$\lambda(t)$ such that $\left\{\ve{u}_{[\lambda(t)]}: t\in [0, T_{+})\right\}$ 
is pre-compact in $\dot{H}^{1}(\R^{6})$. Suppose that, by contradiction, $T_{+}< +\infty$. Following the same argument developed in Case 1 in \cite[Proposition 5.3]{KenigMerle2006} we get that
\begin{align}\label{LamdaI}
	\lim_{t\to T_{+}}\lambda(t)=+\infty.
\end{align}
Now, define for $R>0$,
\[
z_{R}(t)=\int_{\R^{6}}[2\kappa |u(t,x)|^{2}+|v(t,x)|^{2}]\xi\(  \tfrac{x}{R} \)dx
\qtq{for $t\in [0, T_{+})$}
\]
with $\xi=1$ if $|x|\leq 1$ and $\xi=0$ if $|x|\geq 2$. Since 
\[
z^{\prime}_{R}(t)=\tfrac{2\kappa}{R}\IM\int_{\R^{6}}(2\overline{u}\nabla u + \overline{v}\nabla v)\cdot (\nabla\xi)\(  \tfrac{x}{R} \),
\]
it follows by Hardy and Sobolev inequalities, and $H(\ve{u(t)})\leq H(\Q)$ that $|z^{\prime}_{R}(t)|\leq C_{0}$. Thus, applying the fundamental theorem of calculus on $[t, T]\subset [0, T_{+})$ we have
\begin{align}\label{InteF}
	|z_{R}(t)-z_{R}(T)|\leq C_{0}|t-T|.
\end{align}
By compactness \eqref{CompactX}, we infer that that for each $\epsilon>0$, there exists $\rho_{\epsilon}=\rho(\epsilon)>0$
such that
\begin{equation}\label{CritialE}
\int_{|x|>\tfrac{\rho_{\epsilon}}{\lambda(t)}}|u(t,x)|^{3}+|v(t,x)|^{3}\,dx\ll\epsilon.
\end{equation}
Thus, by  \eqref{LamdaI} and  \eqref{CritialE} 
we find
\[
\lim_{t\to T_{+}}z_{R}(t)=0.
\]
From \eqref{InteF} we have $|z_{R}(t)|\leq  C_{0}|t-T_{+}|$.  Letting $R\to+\infty$ we obtain $\ve{u}(t)\in L^{2}$ and
$\|\ve{u}(t)\|^{2}_{L^{2}}\leq C_{0}|t-T_{+}|$. In particular,  by the conservation of the $L^{2}$ norm, we get $\ve{u}_{0}=0$, which is a contradiction with $E(\ve{u})= E(\Q)>0$. Therefore, $T_{+}=+\infty$. The same result holds for negative time.
\end{proof}

\begin{lemma}[Convergence in the ergodic mean]\label{ZeroVirial}
Suppose $\ve{u}$ is a radial solution of \eqref{SNLS} satisfying the assumptions of Proposition~\ref{CompacDeca} 
(cf. \eqref{PropCon11} and \eqref{ScaTa}). 
Then 
\begin{align}\label{Cdelta}
	\lim_{T\to +\infty}\frac{1}{T}\int^{T}_{0}\delta(t)dt=0.
\end{align}
\end{lemma}
\begin{proof} 
Since $|\nabla w_{R}| \lesssim \tfrac{R^{2}}{|x|}$,  by Hardy inequality we see that 
\[
|I_{R}[\ve{u}](t)|\leq C_{\ast} R^{2}
\]
for some $C_{\ast}>0$.
With $\epsilon>0$ given and $R>0$ to be determined below, we write
(cf. Lemma~\ref{VirialIden}) 
\begin{align*}
	\frac{d}{d t}I_{R}[u]&=F_{\infty}[\ve{u}(t)]\\
	                     &+F_{R}[\ve{u}(t)]-F_{\infty}[\ve{u}(t)].
\end{align*}
Observe that by the Pohozaev's identity, $6E(\Q)=H(\Q)$. As $E(\ve{u})=E(\Q)$ we obtain
\[
F_{\infty}[\ve{u}(t)]=8\kk[2H(\ve{u})-3P(\ve{u})]=8\kk[6E(\ve{u})-H(\ve{u})]=8\kk\delta(t).
\]
Therefore,
\[
\frac{d}{d t}I_{R}[\ve{u}]=8\kk\delta(t)+[F_{R}[\ve{u}(t)]-F_{\infty}[\ve{u}(t)]].
\]
Now, notice that
\begin{align*}
F_{R}[\ve{u}(t)]-F_{\infty}[\ve{u}(t)]&
=\int_{|x|\geq R}(-2\kk \Delta \Delta w_{R})[|u|^{2}+\tfrac{\kk}{2}|v|^{2}]-
2\kk\Delta[w_{R}(x)]\RE(\overline{v}u^{2}) \, dx\\
&+16\kk \RE\int_{|x|\geq R}[|\nabla u|^{2}+\tfrac{\kk}{2}|\nabla v|^{2}+\tfrac{3}{2}\overline{v}u^{2}]dx\\
&+8\kk\RE\int_{|x|\geq R} [\overline{u_{j}} u_{k} +\tfrac{\kk}{2}\overline{v_{j}} v_{k}]\partial_{jk}[w_{R}]dx.
\end{align*}
By compactness in $\dot{H}^{1}(\R^{6})$ (see \eqref{CompactX}), we infer that there exists $C_{\epsilon}>0$ 
such that
\[
\sup_{t\geq 0}\int_{|x|>\tfrac{C_{\epsilon}}{\lambda(t)}}[|\nabla u|^{2} +|\nabla v|^{2}+ | u|^{3} +| v|^{3}](t,x)dx  \ll \epsilon.
\]

Thus, using the conditions on the weight $w_{R}$ specified in Section~\ref{S1:preli}, by H\"older inequality we have
for $R\geq \tfrac{C_{\epsilon}}{\lambda(t)}$,
\[
|F_{R}[\ve{u}(t)]-F_{\infty}[\ve{u}(t)]|\leq \epsilon.
\]
Next, we need the following result.
\begin{claim}\label{Claim11}
\begin{align}\label{Claim1}
	\lim_{t\to +\infty}\sqrt{t}\lambda(t)=+\infty.
\end{align}
\end{claim}
Assuming the claim for a moment, we infer that there exists $t_{0}\geq 0$ such that for all $t\geq\geq t_{0}$ we have
\[
\lambda(t)\geq \tfrac{M_{0}}{\sqrt{t}},
\]
where we choose $M_{0}$ such that
\[
M_{0}\epsilon_{0}\geq C_{\epsilon}\qtq{with}
\epsilon^{2}_{0}:=\tfrac{\epsilon}{2C_{\ast}}.
\]
We put $R:=\epsilon_{0}\sqrt{T}$ for $T\geq t_{{0}}$. Thus, for $t\in[t_{0}, T]$, we infer that
\[
R\geq \epsilon_{0}\sqrt{T}\frac{M_{0}}{\sqrt{t}\lambda(t)}=\frac{\sqrt{T}}{\sqrt{t_{0}}}\frac{M_{0}\epsilon_{0}}{\lambda(t)}
\geq \frac{C_{\epsilon}}{\lambda(t)}.
\]
Applying the fundamental theorem of calculus on $[t_{0}, T]$ and collecting the estimates above, we get
\[
8\frac{\kk}{T}\int^{T}_{t}\delta(t)\,dt\leq 2C_{\ast}\frac{R^{2}}{T}+\epsilon\frac{(T-t_{0})}{T}\leq 2\epsilon.
\]
Taking to the limit $T\to +\infty$ and then $\epsilon\to 0$ we find
\[
\lim_{T\to +\infty}\frac{1}{T}\int^{T}_{0}\delta(t)dt=0.
\]

\begin{proof}[{Proof of Claim~\ref{Claim11}}]
Assume by contradiction that \eqref{Claim1} does not hold. Then there exists $s\in [0, +\infty)$ such that
$\lim_{t_{n}\to+\infty}\sqrt{t_{n}}\lambda(t_{n})=s$. In particular,
\begin{align}\label{LamZer}
	\lim_{t_{n}\to+\infty}\lambda(t_{n})=0.
\end{align}
Define
\[
\ve{w}_{n}(\tau, y)=
\lambda(t_{n})^{-2}\ve{u}\( t_{n}+\tfrac{\tau}{\lambda(t_{n})^{2}},\tfrac{y}{\lambda(t_{n})}\).
\]
By compactness we see that there exists $\ve{w}_{0}\in \dot{H}^{1}$ so that $\ve{w}_{n}(0)\to \ve{w}_{0}$ in $\dot{H}^{1}$ as $n\to \infty$.
Since $E(\ve{u}_{0})=E(\Q)$ and $H(\ve{u}(t_{n})<H(\Q)$, it follows that $E(\ve{w}_{0})=E(\Q)$ and $H(\ve{w}_{0})\leq H(\Q)$.
Then Lemma~\ref{GlobalW} implies that the solution $\ve{w}(t)$ to \eqref{SNLS} with initial data $\ve{w}_{0}$ is global and $E(\ve{w}(t))=E(\Q)$ for all $t\in \R$. As $-\sqrt{t_{n}}\lambda(t_{n})\to -s$, we infer that (by stability theory)
\[
\lambda(t_{n})^{-2}\ve{u}_{0}\(\tfrac{y}{\lambda(t_{n})}\)
=\ve{w}_{n}(-t_{n}\lambda(t_{n})^{2}, y)\to \ve{w}(-s^{2},y).
\]
But then, by \eqref{LamZer} we have
\[
\lambda(t_{n})^{-2}\ve{u}_{0}\(\tfrac{y}{\lambda(t_{n})}\)\rightharpoonup 0
\qtq {in $\dot{H}^{1}$,}
\]
which is a contradiction with $E(\ve{w}(-s^{2}))=E(\Q)>0$. This completes the proof of claim. 
\end{proof}
\end{proof}

As an immediate consequence of the lemma above, we obtain the following result.
\begin{lemma}\label{ZeroVirial11}
Let $\ve{u}$ be a radial solution of \eqref{SNLS} satisfying the assumptions of Proposition~\ref{CompacDeca}. 
Then there exists a sequence $t_{n}\to\infty$ such that 
\[
\lim_{n\to +\infty}\delta(t_{n})=0.
\]
\end{lemma}

Let $\mathbf{u}$ be a solution of \eqref{SNLS}. We consider $\delta_0 > 0$ and the modulation parameters $\theta(t)$, $\mu(t)$, and $\alpha(t)$ given by Lemma~\ref{ExistenceFree}, which are defined for all $t \in I_0$. In what follows, we identify $\mu(t)$ with the scaling parameter $\lambda(t)$ from Lemma~\ref{ExistenceFree}. 
From the decomposition \eqref{DecomUFree} and the estimate \eqref{EstimateOne}, we deduce the existence of a constant $C_0 > 0$ satisfying:
\begin{align*}
\int_{\mu(t)\leq|x|\leq 2\mu(t)}\left[|\nabla u(t,x)|^{2}+|\nabla v(t,x)|^{2}\right]dx 
\geq \int_{1\leq|x|\leq 2}|\nabla Q|^{2}-C_{0}\delta(t), \quad \text{for all } t\in I_{0}.
\end{align*}

Thus, for $\delta_{0}>0$ sufficiently small and $t\in I_{0}$, there exists $\epsilon>0$ such that
\begin{align*}
\int_{\frac{\mu(t)}{\lambda(t)}\leq|x|\leq\frac{2\mu(t)}{\lambda(t)}}
\left[\tfrac{1}{\lambda(t)^{6}}\left|\nabla u\left(t,\tfrac{x}{\lambda(t)}\right)\right|^{2}+\tfrac{1}{\lambda(t)^{6}}\left|\nabla v\left(t,\tfrac{x}{\lambda(t)}\right)\right|^{2}\right]dx \geq \epsilon.
\end{align*}

Since $\left\{\ve{u}_{[\lambda(t)]}: t\in [0, +\infty)\right\}$ is pre-compact in $\dot{H}^{1}(\R^{6})$, it follows that $|\mu(t)|\sim|\lambda(t)|$ for $t\in I_{0}$.

As a consequence, we may modify $\lambda(t)$ such that $\left\{\ve{u}_{[\lambda(t)]}: t\in [0, +\infty)\right\}$ remains pre-compact in $\dot{H}^{1}$ with
\begin{align}\label{Kcp}
    \lambda(t)=\mu(t) \quad \text{for all } t\in I_{0}.
\end{align}


\begin{lemma}\label{Lemma11}
There exists a constant $C=C(\delta_{1})>0$ such that for any interval $[t_{1}, t_{2}]\subset [0, \infty)$ we have
\begin{equation}\label{BoundT1}
\int^{t_{2}}_{t_{1}}\delta(t)dt\leq C\sup_{t\in[t_{1},t_{2}]}\tfrac{1}{\lambda(t)^{2}}
\left\{\delta(t_{1})+\delta(t_{2})\right\}.
\end{equation}
\end{lemma}
\begin{proof}
Let $R>1$, which will be determined later. We use the localized virial identities (cf. Lemma \ref{VirialModulate})
with the $\chi(t)$ satisfying
\[
\chi(t)=
\begin{cases}
1& \quad \delta(t)<\delta_{0} \\
0& \quad \delta(t)\geq \delta_{0}.
\end{cases}
\]
From  Lemma \ref{VirialModulate} we have (recall that $F_{\infty}[u(t)]=8\kk\delta(t)$)
\begin{equation}\label{VirilaX}
\frac{d}{dt}I_{R}[u(t)]=F_{\infty}[u(t)]+\EE(t)= 8\kk\delta(t)+\EE(t)
\end{equation}
with
\begin{equation}\label{Error11}
\EE(t)=
\begin{cases}
F_{R}[u(t)]-F_{\infty}[u(t)]& \quad  \text{if $\delta(t)\geq \delta_{0}$}, \\
F_{R}[u(t)]-F_{\infty}[u(t)]-\K[u(t)]& \quad \text{if $\delta(t)< \delta_{0}$},
\end{cases}
\end{equation}
where
\begin{equation}\label{Error22}
\K(t)=F_{R}[\Q_{[\theta(t), \lambda(t)]}]-F_{\infty}[\Q_{[\theta(t), \lambda(t)]}].
\end{equation}

Next we assume the following claims for a moment to conclude the proof.

\textit{{Claim I.}} For $R>1$,  we have
\begin{align}\label{EstimateV11}
	|I_{R}[u(t_{j})]|\lesssim \frac{R^{2}}{\delta_{0}}\delta(t_{j}) \quad& \text{if $\delta(t_{j})\geq \delta_{0}$ for $j=1$, $2$},\\
	\label{EstimateV22}
	|I_{R}[u(t_{j})]|\lesssim R^{2} \delta(t_{j}) \quad &\text{if $\delta(t_{j})< \delta_{0}$ for $j=1$, $2$}.
\end{align}

\textit{{Claim II.}} For $\epsilon>0$, there exists $\rho_{\epsilon}=\rho(\epsilon)>0$ such that if
$R=\rho_{\epsilon}\sup_{t\in [t_{1}, t_{2}]}\tfrac{1}{\lambda(t)}$ we have
\begin{align}\label{EstimateE11}
	|\EE(t)|\leq \frac{\epsilon}{\delta_{0}}\delta(t) \quad &\text{uniformly for $t\in [t_{1}, t_{2}]$ and $\delta(t)\geq \delta_{0}$},\\
	\label{EstimateE22}
|\EE(t)|\leq \epsilon \delta(t)\quad &\text{uniformly for $t\in [t_{1}, t_{2}]$ and $\delta(t)< \delta_{0}$}.
\end{align}

Integrating \eqref{VirilaX} on $[t_{1}, t_{2}]$ and combining the estimates \eqref{EstimateV11}, \eqref{EstimateV22}, \eqref{EstimateE11} 
and \eqref{EstimateE22} together  we have
\[
\int^{t_{2}}_{t_{1}}\delta(t)dt\lesssim 
\frac{\rho_{\epsilon}}{\delta_{0}}\sup_{t\in [t_{1}, t_{2}]}\frac{1}{\lambda(t)^{2}}(\delta(t_{1})+\delta(t_{2}))
+\Big(\frac{\epsilon}{\delta_{0}}+\epsilon\Big)\int^{t_{2}}_{t_{1}}\delta(t)dt.
\]
 Choosing $\epsilon=\epsilon(\delta_{0})$ sufficiently small
we obtain \eqref{BoundT1}.

Thus, it remains to establish the above claims.
\begin{proof}[{Proof of Claim I}] First, suppose that $\delta(t_{j})\geq \delta_{0}$. Then by Hardy inequality we have
\[
|I_{R}[u(t)]|
\lesssim R^{2}\|u\|^{2}_{L^{\infty}_{t}\dot{H}^{1}}\lesssim_{Q} \frac{R^{2}}{\delta_{0}}\delta(t_{j}),
\]
which implies \eqref{EstimateV11}. If instead $\delta(t_{j})< \delta_{0}$, by using the fact that $Q$ is real we get
\begin{align*}
|I_{R}[u(t_{j})]|&\leq \left|2\IM\int_{\R^{6}}\nabla w_{R}(2\overline{u} \nabla u-2\sqrt{\kk}e^{-i\theta(t_{j})}Q_{[\lambda(t_{j})]} 
\nabla [\sqrt{\kk} e^{i\theta(t_{j})}Q_{[\lambda(t_{j})]}])
dx\right|\\
&+
\left|2\IM\int_{\R^{6}}\nabla w_{R}(\overline{v} \nabla v-e^{-2i\theta(t_{j})}Q_{[\lambda(t)]} 
\nabla [e^{2i\theta(t_{j})}Q_{[\lambda(t_{j})]}])dx\right|\\
&\lesssim
R^{2}[\|u\|_{L^{\infty}_{t}\dot{H}^{1}_{x}}+\sqrt{\kk}\|Q\|_{\dot{H}^{1}}]
\|u(t_{j})-e^{i\theta (t_{j})}\sqrt{\kk}Q_{[\lambda(t_{j})]}\|_{\dot{H}^{1}}\\
&+R^{2}[\|v\|_{L^{\infty}_{t}\dot{H}^{1}_{x}}+\|Q\|_{\dot{H}^{1}}]
\|v(t_{j})-e^{i\theta (t_{j})}Q_{[\lambda(t_{j})]}\|_{\dot{H}^{1}}\\
&\lesssim_{Q}R^{2} \delta(t_{j}),
\end{align*}
where in the last inequality we have used estimate \eqref{EstimateOne}.
\end{proof}

\begin{proof}[{Proof of Claim II}]
Assume that $\delta(t)\geq \delta_{0}$. By using \eqref{CompactX} we infer that for each $\epsilon>0$, there exists 
$\rho_{\epsilon}=\rho (\epsilon)>0$  such that
\begin{equation}\label{CompactAgain}
\int_{|x|>\tfrac{\rho_{\epsilon}}{\lambda(t)}}[|\nabla u(t,x)|^{2}+|\nabla v(t,x)|^{2}]dx\ll\epsilon.
\end{equation}
We set
\[
R:=\rho_{\epsilon}\sup_{t\in [t_{1},t_{2}]}\tfrac{1}{\lambda(t)}.
\]
Using the same argument developed above in Lemma~\ref{ZeroVirial} we have
\[
 |F_{R}[u(t)]-F_{\infty}[u(t)]|\leq \epsilon\leq \tfrac{\epsilon}{\delta_{0}} \delta(t)
\quad \text{for every $t\in [t_{1}, t_{2}]$ with $\delta(t)\geq \delta_{0}$}.
\]
This implies estimate \eqref{EstimateE11}.

Next, suppose  $\delta(t)< \delta_{0}$. In order to simplify the notation, 
we put $Q_{1}(t)=e^{i\theta}Q_{\lambda(t)}$ and $Q_{2}(t)=e^{2i\theta}Q_{\lambda(t)}$.
By definition of $\EE(t)$, given in \eqref{Error11}, we can write 
\begin{align}\label{Decomp11}
\EE(t)&=16\kk \int_{|x|\geq R}[(|\nabla u|^{2}+\tfrac{\kk}{2}|\nabla v|^{2})-(|\nabla (\sqrt{\kk} Q_{1})|^{2}+\tfrac{\kk}{2}|\nabla Q_{2}|^{2})]dx
\\\label{Decomp12}
&-16\kk \int_{|x|\geq R}[(\tfrac{3}{2}\overline{v}u^{2})-(\tfrac{3\kk}{2}e^{{i\theta}}Q^{3})]dx\\\label{Decomp33}
&+\int_{|x|\geq R}(-2\kk \Delta \Delta w_{R})[(|u|^{2}+\tfrac{\kk}{2}|v|^{2})-(|\sqrt{\kk}Q_{1}|^{2}+\tfrac{\kk}{2}|Q_{2}|^{2})]dx\\\label{Decomp44}
&-\int_{|x|\geq R}2\kk\Delta[w_{R}(x)](\overline{v}u^{2}-\kk e^{i\theta}Q^{3})dx\\\label{Decomp55}
&+8\kk\RE\int_{|x|\geq R}
 [(\overline{u_{j}} u_{k} +\tfrac{\kk}{2}\overline{v_{j}} v_{k})-
(\overline{\kk\partial_{j}Q_{1}(t)} \partial_{k}Q_{1}(t) +\tfrac{\kk}{2}\overline{\partial_{j}Q_{2}(t)} \partial_{2}Q_{2}(t))]\partial_{jk}[w_{R}]dx.
\end{align}
for all $t\in[t_{1}, t_{2}]$ such that $\delta(t)<\delta_{1}$.
By Hardy and Sobolev inequalities we infer that \eqref{Decomp11}--\eqref{Decomp55} can be estimated by terms of the form
\[
[\|u(t)\|^{\alpha}_{\dot{H}_{x}^{1}(|x|\geq R)}+\|v(t)\|^{\alpha}_{\dot{H}_{x}^{1}(|x|\geq R)}+\|Q_{[\lambda(t)]}\|^{\alpha}_{\dot{H}_{x}^{1}(|x|\geq R)}]
\|\ve{u}(t)-\Q_{\theta(t), \lambda(t)}\|_{\dot{H}_{x}^{1}(|x|\geq R)},
\]
where $\alpha\in \left\{1,2\right\}$. As $\delta(t)<\delta_{1}$, by Proposition~\ref{ModilationFree} and \eqref{CompactAgain},
above term is further estimate by
\[
[\|u(t)\|^{\alpha}_{\dot{H}_{x}^{1}(|x|\geq \tfrac{\rho_{\epsilon}}{\lambda(t)})}+\|v(t)\|^{\alpha}_{\dot{H}_{x}^{1}(|x|\geq \tfrac{\rho_{\epsilon}}{\lambda(t)})}+\|Q\|^{\alpha}_{\dot{H}_{x}^{1}(|x|\geq \rho_{\epsilon})}]\delta(t)
\lesssim \epsilon\delta(t),
\]
for $\rho_{\epsilon}$ sufficiently large depending of $Q$. This completes the proof of Claim II.
\end{proof}

\end{proof}

\begin{proposition}[Control of the variations of the parameter $\lambda(t)$]\label{Spatialcenter}
Let $[t_{1}, t_{2}]$ be an interval of $(0, \infty)$ with $t_{1}+\tfrac{1}{\lambda(t_{1})}\leq t_{2}$. Then there exists  $C>0$ such that
\begin{equation}\label{BoundCenter}
\left|\frac{1}{\lambda(t_{2})^{2}}-\frac{1}{\lambda(t_{1})^{2}}\right|\leq C_{0}\int^{t_{2}}_{t_{1}}\delta(t)\, dt.
\end{equation}
\end{proposition}
\begin{proof}
\textsl{Step 1.} There exists a constant $C_{1}$ such that
\begin{equation}\label{step11}
\tfrac{\lambda(s)\check{}}{\lambda(t)}+\tfrac{\lambda(t)}{\lambda(s)}\leq C_{1} \quad \text{for all $t$, $s\geq 0$ such that $|t-s|\leq \tfrac{1}{\lambda(s)^{2}} $}.
\end{equation}
Indeed, suppose that $s_{n}$, $t_{n}$ obey 
\begin{align}\label{CBound}
	|t_{{n}}-s_{n}|\leq \tfrac{1}{\lambda(s_{n})^{2}} \qtq{but}
\tfrac{\lambda(s_{n})^{2}}{\lambda(t_{n})^{2}}+\tfrac{\lambda(t_{n})^{2}}{\lambda(s_{n})^{2}}\to \infty.
\end{align}
Up subsequence, we can suppose that
\[
\lim_{t\to \infty}\lambda(s_{n})^{2}(t_{n}-s_{n})=\tau_{0}\in[-1, 1].
\]
Consider the solution of \eqref{SNLS}
\[
\ve{v}_{n}(\tau,y)=\lambda(s_{n})^{-2}\ve{u}\( \tfrac{\tau}{\lambda(s_{n})^{2}}+s_{n},\tfrac{y}{\lambda(s_{n})}  \).
\]
We can use compactness to find $\ve{v}_{0}\in \dot{H}^{1}$ such that
\[
\ve{v}_{n}(0,y)\to \ve{v}_{0}(y) \qtq{in $\dot{H}^{1}$ as $n\to \infty$}.
\]
Then $E(\ve{v})=E(\Q)$ and $H(\ve{v}_{0})\leq H(\Q)$. Consider $\ve{v}$ be the solution of \eqref{SNLS} with initial data $\ve{v}_{0}$.
Lemma~\ref{GlobalW} implies that $\ve{v}$ is globally defined and
\[
\ve{w}_{n}(y)=\ve{v}_{n}(\lambda(s_{n})^{2}(t_{n}-s_{n}), y)=
\lambda(s_{n})^{-2}\ve{u}\( t_{n},\tfrac{y}{\lambda(s_{n})}\)\to \ve{v}(\tau_{0}, y).
\]
On the other hand, by compactness we see that
\[
\tfrac{1}{\lambda(t_{n})^{2}}\ve{u}\( t_{n},\tfrac{y}{\lambda(t_{n})}\)=
\tfrac{\lambda(s_{n})^{2}}{\lambda(t_{n})^{2}}\ve{w}_{n}\(\tfrac{\lambda(s_{n})}{\lambda(t_{n})}y\)\to \varphi\neq 0
\]
in $\dot{H}^{1}$, which implies that $\tfrac{\lambda(s_{n})}{\lambda(t_{n})}$ and 
$\tfrac{\lambda(t_{n})}{\lambda(s_{n})}$ are bounded, contradicting \eqref{CBound}.

\textsl{Step 2.}
There exists $\delta_{1}>0$ such that either
\begin{equation}\label{MinMax}
\inf_{t\in [T, T+\tfrac{1}{\lambda(T)^{2}}]}\delta(t)\geq \delta_{1} \quad \text{or}\quad
\sup_{t\in [T,T+\tfrac{1}{\lambda(T)^{2}}]}\delta(t)<\delta_{0}\quad \text{for any  $T\geq 0$}.
\end{equation}
Indeed, \eqref{MinMax} is proved by contradiction. Suppose that there exist $t_{n}^{\ast}\geq 0$ and two sequences
$t_{n}$, $t^{\prime}_{n}\in  [t_{n}^{\ast}, t_{n}^{\ast}+\tfrac{1}{\lambda(t_{n}^{\ast})^{2}}]$ such that
\begin{align}\label{ContraStep2}
&\delta(t_{n})\to 0 \quad \text{and}\quad \delta(t^{\prime}_{n})\geq \delta_{1} \quad \text{as $n\to \infty$}.
\end{align}
By Step 1, we see that $\tfrac{\lambda(t_{n})}{\lambda(t^{\ast}_{n})}\leq C$ for some $C>0$. Thus, possibly for a subsequence only,
we may assume that
\begin{align}\label{ContraLimit}
& \lambda(t_{n})^{2}(t_{n}-t^{\prime}_{n})\to t^{\ast}\in[-C,C].
\end{align}
Consider
\[
\ve{v}_{n}(\tau,y)=\lambda(t_{n})^{-2}\ve{u}\( \tfrac{\tau}{\lambda(t_{n})^{2}}+t_{n},\tfrac{y}{\lambda(t_{n})}  \).
\]
Now, since $\delta(t_{n})\to 0$, by compactness we obtain that there exists (cf. Proposition~\ref{TheGN})
$\lambda_{0}>0$ such that 
\begin{equation}\label{Step2Conver}
\text{$\ve{v}_{n}(0, \cdot)\to \Q_{[\theta_{0},\lambda_{0}]}$ strongly in $\dot{H}^{1}$ as $n\to \infty$}.
\end{equation}
Therefore, putting together \eqref{ContraLimit} and \eqref{Step2Conver} we get
\[
\lambda(t_{n})^{-2}\ve{u}\(t^{\prime}_{n},\tfrac{y}{\lambda(t_{n})}  \)=
\ve{v}_{n}(\lambda(t_{n})^{2}(t_{n}-t^{\prime}_{n}), y)\to \Q_{[\theta_{0},\lambda_{0}]},
\]
which is a contradiction with \eqref{ContraStep2}.

\textsl{Step 3.} Next we show
\begin{align}\label{acD}
	0\leq t_{1}\leq \tilde{t_{1}}\leq \tilde{t_{2}}\leq t_{2}=t_{1}+\frac{1}{C^{2}_{1}\lambda(t_{2})^{2}}
\Rightarrow 
\left|\tfrac{1}{\lambda(\tilde{t_{2}})^{2}}-\tfrac{1}{\lambda(\tilde{t_{1}})^{2}}\right|\leq C\int^{t_{2}}_{t_{1}}\delta(t)\,dt.
\end{align}
for some $C>0$, where $C_1$ is defined in \eqref{step11}. Indeed, we can assume that $\sup_{t\in [t_{1}, t_{2}]}\delta(t)<\delta_{0}$ or
$\inf_{t\in [t_{1}, t_{2}]}\delta(t)\geq \delta_{1}$. In the first case, we obtain \eqref{acD} by time-integration 
and $\left|\tfrac{\lambda^{\prime}({t})}{\lambda(t)^{3}}\right|\lesssim \delta(t)$ for $\delta(t)<\delta_{0}$ (cf. \eqref{EstimateFree}). 
In the second case, we have $\int^{t_{2}}_{t_{1}}\delta(t)\,dt\geq \int^{t_{2}}_{t_{1}}\delta_{1}\,dt$ and
\[
|\tilde{t_{1}}-\tilde{t_{2}}|\leq \tfrac{1}{C^{2}_{1}\lambda(t_{1})^{2}}\leq \tfrac{1}{\lambda(\tilde{t_{1}})^{2}}.
\]
Thus, by Step 1 we get (note that $C_{1}\geq1$)
\[
\left|\tfrac{1}{\lambda(\tilde{t_{2}})^{2}}-\tfrac{1}{\lambda(\tilde{t_{1}})^{2}}\right|
\leq\tfrac{2 C^{2}_{1}}{\lambda(\tilde{t_{1}})^{2}}
\leq \tfrac{2 C^{4}_{1}}{\lambda({t_{1}})^{2}}
=2C^{5}_{1}|t_{2}-t_{1}|\leq \tfrac{2C^{5}_{1}}{\delta_{1}}\int^{t_{2}}_{t_{1}}\delta(t)dt
\]
To complete the proof, it remains divide $[t_{1}, t_{2}]$ into small intervals and stick together the inequalities above to obtain 
\eqref{BoundCenter}.
\end{proof}

\begin{proof}[{Proof of Proposition~\ref{CompacDeca}}]
By Lemmas~\ref{ZeroVirial11} and \ref{Lemma11}, and Proposition~\ref{Spatialcenter}, and using the same argument developed in  \cite[Lemma 6.9]{MIWUGU2015}, one can show that 
$\tfrac{1}{\lambda(t)^{2}}$ is bounded on $[0,\infty)$.  Indeed, we use the standard localized virial to find a sequence $t_n\to\infty$ with $\delta(t_n)\to 0$ (cf. Lemma~\ref{ZeroVirial}). Then,  using the modulated virial and the fact that the integral of $\delta$ controls the variation of
 $\tfrac{1}{\lambda(\cdot)^{2}}$ (cf. Proposition~\ref{Spatialcenter}), we can obtain that $\tfrac{1}{\lambda(t)^{2}}\lesssim \tfrac{1}{\lambda(t_{N})^{2}}$ for all $t\geq t_N$ for some sufficiently large $N$.

Applying Lemma~\ref{Lemma11}, we then infer that there exists a constant $C>0$ so that
\[
\int^{s}_{T}\delta(t)\,dt\leq C\left\{\delta(T)+\delta(s)\right\}
\qtq{with $[T, s]\subset [0, \infty]$.}
\]
Applying this with a sequence $t_n\to\infty$ such that $\delta(t_n)\to 0$, we see that $\int^{\infty}_{T}\delta(t)\,dt\leq C\delta(T)$ for all $T\geq 0$.  Gronwall's lemma then implies 
\begin{align}\label{PIN}
	\int^{\infty}_{T}\delta(t)\,dt\leq C e^{-c T}.
\end{align}
for some $C$, $c>0$. Combining this inequality with estimate \eqref{alfaE} and employing the same argument as in Proposition~\ref{SupercriQ} (cf. \eqref{deltaN}) below, one can show that
\begin{align}\label{DeltaZ}
	\lim_{t\to \infty}\delta(t)=0.
\end{align}
In view of \eqref{DefH}, to conclude the proof of proposition, it is sufficient to show that there exist $\lambda_{\infty}\in (0, +\infty)$  and $\theta\in \R$
such that
\begin{align}\label{EstimaC}
	\delta(t)+|\alpha(t)|+\|\ve{h}(t)\|_{\dot{H}^{1}}+|\theta(t)-\theta_{\infty}|
	+\left|\frac{1}{\lambda(t)^{2}}-\frac{1}{\lambda_{\infty}^{2}}\right|\leq C e^{-c T}.
\end{align}
Note that by Lemma~\ref{BoundCenter} and \eqref{PIN} we have
there exists  $C>0$ such that
\[
\left|\tfrac{1}{\lambda(t_{2})^{2}}-\tfrac{1}{\lambda(t_{1})^{2}}\right|\leq C e^{-ct_{2}}
\qtq{for $t_{1}+\tfrac{1}{\lambda(t_{1})}\leq t_{2}$.}
\]
Then, Cauchy criteria of convergence implies that there exists $\lambda_{\infty}\in (0, +\infty]$ such that
\[
\left|\tfrac{1}{\lambda(t)^{2}}-\tfrac{1}{\lambda_{\infty}^{2}}\right|\leq C e^{-ct}.
\]
We claim that $\lambda_{\infty}<+\infty$. Indeed, suppose by contradiction that $\lambda_{\infty}=+\infty$. 
From Lemma~\ref{ZeroVirial11} there exists $t_{n}\to +\infty$ such that
\[
\delta(t_{0})\leq \frac{1}{2C_{0}} \qtq{and} \delta(t_{n})\to 0
\]
as $n\to \infty$. Here $C_{0}$ is the constant of estimate \eqref{BoundCenter}. Consider $0\leq a\leq b$.
For $n$ large we get $b+\tfrac{1}{\lambda(b)^{2}}<t_{n}$. By estimate \eqref{BoundCenter}, we find
\[
\left|\tfrac{1}{\lambda(b)^{2}}-\tfrac{1}{\lambda(t_{n})^{2}}\right|\leq C_{0}
\sup_{a\leq t\leq t_{n}}\(\frac{1}{\lambda(t)^{2}}\)\left\{\delta(a)+\delta(t_{n})  \right\}.
\]
Thus, we obtain as $n\to+\infty$,
\[
\sup_{t\geq a}\(\frac{1}{\lambda(t)^{2}}\)\leq C_{0}\delta(a)\sup_{t\geq a}\(\frac{1}{\lambda(t)^{2}}\).
\]
Choosing $a=t_{0}$, we get a contradiction. This proves the bound on $\lambda(t)$ in \eqref{EstimaC}.

On the other hand, as $\alpha(t)\sim \delta(t)$, \eqref{DeltaZ} implies that $\lim_{t\to \infty}\alpha(t)=0$.  Therefore, from \eqref{EstimateFree} we get
\[
\delta(t)+\|\ve{h}(t)\|_{\dot{H}^{1}}\sim |\alpha(t)|
\leq C\int^{+\infty}_{t}|\alpha^{\prime}(s)|ds\leq 
C\int^{+\infty}_{t}\lambda(s)^{2}\delta(s)ds\leq Ce^{-ct}.
\]
Finally, since $\int^{\infty}_{T}|\theta^{\prime}(t)|dt\leq Ce^{-cT}$, we see that there exists $\theta_{\infty}\in \R$ such that
$|\theta(t)-\theta_{\infty}|\leq Ce^{-ct}$. Thus we deduce the bound on $\theta(t)$, which concludes the proof. 
\end{proof}

\begin{proof}[{Proof of Corollary~\ref{ClassC}}]
Suposse that $\ve{u}$ satisfies \eqref{PropCon11} and \eqref{Infinity10}. Arguing as above, one can construct $\lambda(t)$ such that the set
 $\left\{\ve{u}_{[\lambda(t)]}(t)): t\in \R\right\}$ is  pre-compact in $\dot{H}^{1}$.  Moreover, we can prove that $\frac{1}{\lambda(t)^{2}}$ is bounded and 
\[
\lim_{t\to -\infty}\delta(t)=\lim_{t\to \infty}\delta(t)=0.
\]
Also, modifying the proof of Lemma~\ref{Lemma11}, we get
\[
\int^{n}_{-n}\delta(t)\,dt\leq C(\delta(n)+\delta(-n))\quad \text{for all $n\in \N$}.
\]
Sending $n\to \infty$, we find that $\delta(t)\equiv 0$, which is a contradiction with \eqref{PropCon11}.\end{proof}

\section{Convergence for supercritical threshold solution}\label{S: SuperThres}

This section is devoted to establishing the following result:

\begin{proposition}\label{SupercriQ}
Fix $\kk>0$. Let $\ve{u}\in H^{1}(\R^{6})\times H^{1}(\R^{6})$ be a radial solution to \eqref{SNLS} satisfying 
\begin{align*}
E(\ve{u}_{0})=E(\Q) \qtq{and} H(\ve{u}_{0})>H(\Q),
\end{align*}
which is globally defined for the positive times. Then there exist $\theta_{0}\in \R$, $\lambda_{0}>0$, $c$, $C>0$ such that
\begin{align}\label{ExpoAbove}
\|\ve{u}(t)-\Q_{[\theta_{0}, \lambda_{0}]}\|_{\dot{H}^{1}}\leq Ce^{-c t}
\qtq{for all $t\geq 0$.}
\end{align}
Moreover, the negative time of existence is finite.
\end{proposition}

The proof of this proposition is based on the following two lemmas.

\begin{lemma}\label{BoundDQ}
With $\ve{u}(t)$ as in Proposition~\ref{SupercriQ}, there exists $R_{1}>0$ such that for $R\geq R_{1}$ we have
\begin{align}\label{DoDeri}
\tfrac{d}{dt}I_{R}[\ve{u}(t)]\leq -4\kk\delta(t) \qtq{for all $t\geq 0$.}
\end{align}
\end{lemma}
\begin{proof}
For $R>0$ to be determined below, we write (cf. Lemma~\ref{VirialIden} ) 
\begin{align*}
	\frac{d}{d t}I_{R}[\ve{u}]&=F_{\infty}[\ve{u}(t)] + F_{R}[\ve{u}(t)]-F_{\infty}[\ve{u}(t)].
\end{align*}
Since
\[
F_{\infty}[\ve{u}(t)]=8\kk[2H(\ve{u})-3P(\ve{u})]=8\kk[6E(\ve{u})-H(\ve{u})]=-8\kk\delta(t),
\]
we obtain
\[
\frac{d}{d t}I_{R}[\ve{u}]=-8\kk\delta(t)+[F_{R}[\ve{u}(t)]-F_{\infty}[\ve{u}(t)]],
\]
where
\begin{align*}
F_{R}[\ve{u}(t)]-F_{\infty}[\ve{u}(t)]&
=\int_{|x|\geq R}(-2\kk \Delta \Delta w_{R})[|u|^{2}+\tfrac{\kk}{2}|v|^{2}]-
2\kk\Delta[w_{R}(x)]\overline{v}u^{2}\\
&-16\kk \int_{|x|\geq R}[|\nabla u|^{2}+\tfrac{\kk}{2}|\nabla v|^{2}+\frac{3}{2}\overline{v}u^{2}]\\
&+8\kk\RE\int_{|x|\geq R} [\overline{u_{j}} u_{k} +\tfrac{\kk}{2}\overline{v_{j}} v_{k}]\partial_{jk}[w_{R}]dx.
\end{align*}

\textbf{Step 1.} General bound on $|F_{R}[\ve{u}(t)]-F_{\infty}[\ve{u}(t)]|$.
Choosing suitable $\phi$ such that $\partial^{2}_{r}w_{R}\leq 2$ we infer that
\begin{align*}
&8\kk\RE\int_{|x|\geq R} [\overline{u_{j}} u_{k} +\tfrac{\kk}{2}\overline{v_{j}} v_{k}]\partial_{jk}[w_{R}]dx-
16\kk \int_{|x|\geq R}[|\nabla u|^{2}+\tfrac{\kk}{2}|\nabla v|^{2}]dx\\
&=8\kk\int_{|x|\geq R} [|\nabla u|^{2}+\tfrac{\kk}{2}|\nabla v|^{2}](\partial^{2}_{r}w_{R}- 2)dx\leq 0.
\end{align*}
Thus, by using H\"older  inequality we obtain
\[
|F_{R}[\ve{u}(t)]-F_{\infty}[\ve{u}(t)]|
\lesssim
\int_{|x|\geq R} \tfrac{1}{R^{2}}[|u|^{2}+\tfrac{\kk}{2}|v|^{2}]dx
+\int_{|x|\geq R} [|u|^{3}+|v|^{3}]dx.
\]
Moreover, Strauss Lemma (radial Sobolev inequality) implies
\[
\int_{|x|\geq R} [|u|^{3}+|v|^{3}]dx\leq \tfrac{C}{R^{\frac{5}{2}}}
[\|\nabla u\|^{\frac{1}{2}}_{L^{2}}+\|\nabla v\|^{\frac{1}{2}}_{L^{2}}],
\]
where the constant $C$ depends only on $\|\ve{u}_{0}\|_{L^{2}}$.
Therefore,
\begin{align}\label{GeneralB}
	|F_{R}[\ve{u}(t)]-F_{\infty}[\ve{u}(t)]|
\leq C_{0}\left[ \tfrac{1}{R^{2}} +\tfrac{1}{R^{\frac{5}{2}}} (\delta(t)+H(\Q))^{\frac{1}{4}} \right].
\end{align}

\textbf{Step 2.} Bound on $	|F_{R}[\ve{u}(t)]-F_{\infty}[\ve{u}(t)]|$ when $\delta(t)$ is small.

By \eqref{DecomUFree} we have $\ve{u}_{[\theta(t), \lambda(t)]}=\Q+\ve{V}$, where $\|\ve{V}\|_{\dot{H}^{1}}\sim \delta(t)$.
We claim that
\begin{align}\label{InfP}
	\lambda_{\text{inf}}:=\inf\left\{{\lambda(t), t\geq 0, \delta(t)\leq \delta_{1}}\right\}>0
\end{align}
for $\delta_{1}$ sufficiently small. Indeed, writing $\ve{V}=(v_{1}, v_{2})$, by the mass conservation we see that
\begin{align*}
	M(\ve{u}_{0})&\gtrsim \int_{|x|\leq \tfrac{1}{\lambda(t)}}[|u(x,t)|^{2}+|v(x,t)|^{2}]dx
	=\tfrac{1}{\lambda(t)^{2}}\int_{|x|\leq 1}[|u_{[\theta(t), \lambda(t)]}|^{2}+|v_{[\theta(t), \lambda(t)]}|^{2}]dx\\
	&\gtrsim \tfrac{1}{\lambda(t)^{2}}\(\int_{|x|\leq 1}Q^{2}dx-\int_{|x|\leq 1}[|v_{1}|^{2}+|v_{2}|^{2}]dx\).
\end{align*}
Thus, as
\[
\|\ve{V}(t)\|_{L^{2}(|x|\leq 1)}\lesssim \|\ve{V}(t)\|_{L^{3}(|x|\leq 1)}  
\lesssim \|\ve{V}(t)\|_{\dot{H}^{1}}\lesssim\delta(t)
\]
it follows
\[
\|\ve{u}_{0}\|_{L^{2}}\gtrsim  \tfrac{1}{\lambda(t)^{2}}\(\int_{|x|\leq 1}Q^{2}dx-C\delta^{2}(t)\).
\]
Choosing $\delta_{1}$ sufficiently small, we get \eqref{InfP}.

Next, we set
\[
A_{R}(\ve{u}(t)):=F_{R}[\ve{u}(t)]-F_{\infty}[\ve{u}(t)].
\]
A change of variable implies
\[
|A_{R}(\ve{u}(t))|=|A_{R\lambda(t)}(\ve{V}(t)+\Q)|.
\]
Notice that
\begin{align}\label{IdenQ}
	A_{R\lambda(t)}(\Q)=0, \qtq{and $\|Q\|_{\dot{H}^{1}(|x|\geq r)}\sim \|Q\|_{L^{3}(|x|\geq r)}\sim r^{-2}$ for $r\geq1$.}
\end{align}

By H\"older, Hardy and  Sobolev inequalities and \eqref{IdenQ}, we get for $R\geq1$ (recall that $\ve{V}=(v_{1}, v_{2})$),
\begin{align*}
|A_{R}(\ve{u}(t))|&=|A_{R\lambda(t)}(\Q+\ve{V}(t))|=|A_{R\lambda(t)}(\Q+\ve{V}(t))-A_{R\lambda(t)}(\Q)|\\
&\leq C \int_{|x|\geq R\lambda(t)}[|\nabla v_{1}|^{2}  + |\nabla v_{2}|^{2}+|\nabla Q\cdot \nabla v_{1}| +|\nabla Q\cdot \nabla v_{1}|+|\nabla Q\cdot \nabla v_{2}|] dx\\
&+C\int_{R\lambda(t)\leq|x|\leq 2R\lambda(t)}\tfrac{1}{(R\lambda(t))^{2}}[| v_{1}|^{2}+ | v_{2}|^{2}+| Q| |v_{1}| +| Q| |v_{2}|]dx\\
&+C\int_{|x|\geq R\lambda(t)}[|v_{1}|^{2}|v_{2}|+|v_{1}|^{2}Q+ |v_{1}||v_{2}|Q+ |v_{1}|Q^{2}  +  |v_{2}|Q^{2} ]dx\\
&\leq C\left[ \|\ve{V}\|^{2}_{\dot{H}^{1}}  +\tfrac{1}{(R\lambda(t))^{2}}\|\ve{V}\|_{\dot{H}^{1}} 
+\frac{1}{(R\lambda(t))}\|\ve{V}\|^{2}_{\dot{H}^{1}}+\tfrac{1}{R\lambda(t)^{2}}\|\ve{V}\|^{2}_{\dot{H}^{1}}
+\|\ve{V}\|^{3}_{\dot{H}^{1}}\right]\\
&\leq C_{\ast}\left[ \delta(t)^{2}+\tfrac{1}{R^{2}}\delta(t) \right],
\end{align*}
where the constant $C_{\ast}$ depends only on $\lambda_{\text{inf}}$.

\textbf{Step. 3} Conclusion. To prove \eqref{DoDeri}, it sufficiently to show
\begin{align}\label{ClaimDel}
	|F_{R}[\ve{u}(t)]-F_{\infty}[\ve{u}(t)]|\leq 4\kk\delta(t).
\end{align}
By Step 2, there exists $\delta_{2}>0$ such that if  $\delta(t)\leq\delta_{2}$ and $R\geq R_{1}$, with $R_{1}$ sufficiently large, then
\[
|F_{R}[\ve{u}(t)]-F_{\infty}[\ve{u}(t)]|\leq  C_{\ast}\left[ \delta(t)^{2}+\tfrac{1}{R^{2}}\delta(t) \right]
\leq 4\kk\delta(t).
\]
On the other hand, for $\delta(t)>\delta_{2}$ we consider the function
\[
f_{R}(\delta):=C_{0}\left[ \tfrac{1}{R^{2}} +\tfrac{1}{R^{\frac{5}{2}}} (\delta+H(\Q))^{\frac{1}{4}} \right]-4\kk\delta,
\]
where $C_{0}$ is given by \eqref{GeneralB}. Notice that $f^{\prime\prime}_{R}(\delta)<0$ for all $\delta>0$.
For sufficiently large $R_{2}$ we have $f_{R_{2}}(\delta_{2})\leq 0$ and $f^{\prime}_{R_{2}}(\delta_{2})\leq 0$. This implies $f_{R}(\delta)\leq 0$ for all $\delta\geq \delta_{2}$ and $R\geq R_{2}$. Thus, the bound \eqref{ClaimDel} holds for $R=\max\left\{R_{1}, R_{2}\right\}$.

\end{proof}

\begin{lemma}\label{NeD}
With $\ve{u}(t)$ as in Proposition~\ref{SupercriQ}, there exist $c>0$, $C>0$ and $R_{1}>0$ such that for $R\geq R_{1}$ we have
\begin{align}\label{DoQ11}
\int^{+\infty}_{t}\delta(s)ds \leq Ce^{-ct}\qtq{for all $t\geq 0$.}
\end{align}
\end{lemma}
\begin{proof}
We define
\[
V_{R}(t)=\int_{\R^{6}}(2\kk|u(t,x)|^{2}+|v(t,x)|^{2})dx.
\]
Notice that $\tfrac{d}{dt}V_{R}(t)=I_{R}[\ve{u}]$ (cf. Section~\ref{S1:preli}). From Lemma~\ref{BoundDQ}, we see that
$\tfrac{d^{2}}{dt^{2}}V_{R}(t)=\tfrac{d}{dt}I_{R}[u(t)]\leq -4\kk\delta(t)$ for  $R\geq R_{1}$. 

On the other hand, since $\tfrac{d^{2}}{dt^{2}}V_{R}(t)<0$ and $V_{R}(t)>0$ for all $t\geq0$, it follows that  $I_{R}[u(t)]=\tfrac{d}{dt}V_{R}(t)>0$ for all $t\geq 0$. Thus,
\[
4\kk\int^{T}_{t}\delta(s)ds\leq -\int^{T}_{t}\tfrac{d}{ds}I_{R}[u(s)]ds
=I_{R}[u(t)]-I_{R}[u(T)]\leq I_{R}[u(t)]\leq CR^{2}\delta(t),
\]
where we have used the estimate $I_{R}[u(t)]\leq CR^{2}\delta(t)$ for all $t\geq 0$ (cf. \eqref{EstimateV11} and \eqref{EstimateV22}). By the Gronwall inequality we obtain \eqref{DoQ11}.
 \end{proof}

\begin{proof}[{Proof of Proposition~\ref{SupercriQ}}]

From Lemma~\ref{NeD} we find that there exists $\left\{t_{n}\right\}_{n\in \N}$ with $t_{n}\to +\infty$ such that
$\lim_{n \to \infty}\delta(t_{n})=0$. Fix such $\left\{t_{n}\right\}_{n\in \N}$.
Notice that  
\begin{align}\label{deltaN}
	\lim_{t \to \infty}\delta(t)=0.
\end{align}
 If not, there exists a sequence $\left\{t^{\prime}_{n}\right\}_{n\in \N}$ so that
$\delta(t^{\prime}_{n})\geq \epsilon$ for some $\epsilon\in(0, \delta_{0})$.
Moreover, extracting subsequences of $\left\{t_{n}\right\}_{n\in \N}$ and $\left\{t^{\prime}_{n}\right\}_{n\in \N}$ if necessary, 
we can assume the following:
\[
t_{n}<t^{\prime}_{n}, \quad \delta(t^{\prime}_{n})=\epsilon, \quad \delta(t)<\epsilon \quad \text{for all $t\in [t_{n}, t^{\prime}_{n})$}.
\]
Note that on $[t_{n}, t^{\prime}_{n})$ the parameters $\alpha(t)$, $\theta(t)$ and $\lambda(t)$ are well defined and recall that (cf. \eqref{DecomUFree})
\[
\ve{u}_{[\theta(t), \lambda(t)]}(t)=(1+\alpha(t))\Q+\ve{h}(t).
\] 
We claim that
\begin{align}\label{BounLa}
	\lim_{n\to \infty}\lambda(t_{n})=\lambda_{\infty}\in (0, +\infty).
\end{align}
Indeed, by estimate $\left|\tfrac{\lambda^{\prime}(t)}{\lambda(t)^{3}}\right|\leq C\delta(t)$ and \eqref{DoQ11} we infer that
\begin{equation}\label{BoLamda}
\left|\tfrac{1}{\lambda(t)^{2}}-\tfrac{1}{\lambda(t_{n})^{2}}\right|\leq C_{0}e^{-ct_{n}}
\qtq{for $t\in [t_{n}, t^{\prime}_{n})$.}
\end{equation}
On the other hand, assume that $\lambda_{\infty}=\infty$. Let $r_{0}>0$. By H\"older, Hardy and Sobolev inequalities we obtain
\begin{align*}
	|V_{R}(t_{n})|&\lesssim
r^{4}_{0}H(\Q)+\|u(t_{n})\|^{2}_{{L}^{3}(|x|\geq r_{0})}
+\|v(t_{n})\|^{2}_{{L}^{3}(|x|\geq r_{0})}.
\end{align*}
Since $\ve{u}_{[\theta(t_{n}), \lambda(t_{n})]} \to \Q$ in $\dot{H}^{1}$, we infer  that for each $r_{0}>0$,
\begin{equation}\label{Limuv}
\|u(t_{n})\|^{2}_{{L}^{3}(|x|\geq r_{0})}
+\|v(t_{n})\|^{2}_{{L}^{3}(|x|\geq r_{0})}\to 0
\qtq{as $n\to \infty$.}
\end{equation}
Taking to the limit $n\to \infty$, and then  $r_{0} \to 0$
we find
\[
\lim_{n\to \infty}V_{R}(t_{n})=0.
\]
But then, as $\tfrac{d}{dt}V_{R}>0$ for all $t\geq 0$ (cf. Lemma~\ref{NeD}), it follows that $V_{R}(t)<0$ for $t\geq 0$, which is a contradiction. 
Therefore, $\lambda_{\infty}<\infty$. In particular, by \eqref{BoLamda} we find that $\lambda(t)\leq 2\lambda_{\infty}$ 
on $\cup [t_{n}, t^{\prime}_{n})$.

As $|\alpha^{\prime}(t)|\lesssim \lambda(t)^{2}|\delta(t)|\lesssim|\delta(t)|$ on $\cup [t_{n}, t^{\prime}_{n})$ (cf. \eqref{EstimateFree}), estimate \eqref{DoQ11} implies 
\begin{equation}\label{limnn}
\lim_{n\to \infty}|\alpha(t_{n})-\alpha(t^{\prime}_{n})|=0.
\end{equation}
Since $|\alpha| \sim |\delta|$ (cf. \eqref{EstimateOne}), we deduce
\[
|\alpha(t_{n})| \sim |\delta(t_{n})|\to 0 \quad \text{and} \quad 
|\alpha(t^{\prime}_{n})| \sim |\delta(t^{\prime}_{n})|=\epsilon>0,
\]
which is a contradiction with \eqref{limnn}. Therefore, $\lim_{t \to \infty}\delta(t)=0$. In particular, the parameters $\alpha(t)$, $\lambda(t)$, and $\theta(t)$ are well-defined for large $t$, and from \eqref{InfP}, we deduce that $\lambda_{\infty}>0$.

Next, as (recall that $\lambda(t)\leq 2\lambda_{\infty}$)
\[
\delta(t)+\|\ve{h}(t)\|_{\dot{H}^{1}}\sim |\alpha(t)|
\leq C\int^{+\infty}_{t}|\alpha^{\prime}(s)|ds\leq 
C\int^{+\infty}_{t}\lambda(s)^{2}\delta(s)\,ds\leq Ce^{-ct},
\]
it follows by Lemma~\ref{EstimateFree},
\[
\|\ve{u}-\Q_{[\theta(t), \lambda(t)]}\|_{\dot{H}^{1}}+|\alpha^{\prime}(t)|+|\theta^{\prime}(t)|\leq  Ce^{-ct},
\]
which implies \eqref{ExpoAbove}.

Finally, we show the finite time blow up for negative times. 
Suppose by contradiction that $\ve{u}$ is global in negative time. Writing $\ve{v}(t,x):=\overline{\ve{u}(-t,x)}$,  we have  that  
Lemmas~\ref{BoundDQ} and \ref{NeD} also hold for negative times. In particular, we find
\[
\lim_{t\to \pm\infty}\delta(t)=0.
\]
Thus, we obtain that $\frac{d}{dt}V_{R}(t)\to 0$ as $t\to \pm\infty$ and $\frac{d}{dt}V_{R}(t)> 0$ for all $t\in \R$ (cf. Lemma~\ref{NeD}). Since $\frac{d^{2}}{dt^{2}}V_{R}(t)< 0$ for all $t\in\R$,
we obtain a contradiction.

This completes the proof of proposition.
\end{proof}


\section{Existence of special solutions}\label{S:Existence}
Due to the operator $J$ that appears in the linearized equation \eqref{Decomh}, we cannot directly apply the arguments developed in \cite{DuyMerle2009, DuyckaertsRou2010} to obtain approximate and special solutions converging to the ground state $\Q$. To overcome this difficulty, we define the transformation:
\begin{equation}\label{TRA}
T(\ve{u}) := T(u_{1}, u_{2}) = \left( \frac{u_{1}}{\sqrt{2}}, \frac{u_{2}}{{2}} \right).
\end{equation}
Note that if $\ve{u}$ is a solution to \eqref{SNLS}, then $T(\ve{u})$ is a solution to the system:
\begin{equation}\label{NNLS}
\left\{
\begin{array}{ccc}
i\partial_t u + \Delta u +2\overline{u}v &=& 0, \\
i\partial_t v +  \kk\Delta v+u^{2} &=& 0. \\
\end{array}
\right.	
\end{equation}
Moreover, the ground state for this NLS system \eqref{NNLS} is given by $T(\Q)$ and the associated energy is
\[
E_{N}(\ve{u}(t)):= \tfrac{1}{2} H_{N}(\ve{u}(t))-\RE\int_{\R^6} u_{1}^{2}\overline{u_{2}}dx,
\]
where
\begin{equation}\label{HNN}
 H_{N}(\ve{u}):=\|\nabla u_{1}\|^2_{L^2(\R^6)} + {\kk}\|\nabla u_{2}\|^2_{L^2(\R^6)}.
\end{equation}
Note also that if $\ve{W}$ is solution to \eqref{NNLS}, then we have that $T^{-1}(\ve{W})$ is solution to \eqref{SNLS} with
$E(T^{-1}(\ve{W}))=2E_{N}(\ve{W})$ and $H(T^{-1}(\ve{W}))=2H_{N}(\ve{W})$.

Throughout this section, we will construct approximate and especial solutions to \eqref{NNLS} that converge to $T(\Q)$. Subsequently, by applying $T^{-1}$, we will obtain approximate and especial solutions to \eqref{SNLS} that converge to $\Q$.

We consider a radial solution $\ve{u}(t)$ of \eqref{NNLS} close to $T(\Q)=(\frac{\sqrt{\kk}}{\sqrt{2}}Q, \frac{1}{2}Q)$ and write
\[
\ve{k}(t,x) = \ve{u}(t,x) - T(\Q).
\]
Then, $\ve{k} = (h, g)$ satisfies the equation:
\[
(i\partial_{t}h, i\partial_{t}g) + (\Delta h, {\kk}\Delta g) + B(\ve{h}) = -N(\ve{k}),
\]
where $B(\ve{k}) = B(h, g) = (\overline{h}Q + \sqrt{2\kk}Qg, \sqrt{2\kk}Qh)$ and $N(\ve{k}) = (2\overline{h} g, h^{2})$. Note also that $\ve{k}$ is a solution of the equation:
\begin{equation}\label{De112}
\partial_{t}\ve{k} + \EE \ve{k} = iN(\ve{k}), \quad \text{where} \quad
\EE := \begin{pmatrix}
0 & -E_{I} \\
E_{R} & 0
\end{pmatrix},
\end{equation}
and the operators $E_{I}$ and $E_{R}$ are defined by (acting on $L^2(\R^6;\R)$) 
\begin{align*}
E_{R} &=
\begin{bmatrix}
-\Delta & 0 \\
0 & -{\kk}\Delta 
\end{bmatrix}
-
\begin{bmatrix}
Q & \sqrt{2\kk}Q \\
\sqrt{2\kk}Q & 0 
\end{bmatrix},
\\
E_{I} &=
\begin{bmatrix}
-\Delta & 0 \\
0 & -{\kk}\Delta 
\end{bmatrix}
+
\begin{bmatrix}
Q & -\sqrt{2\kk}Q \\
-\sqrt{2\kk}Q & 0 
\end{bmatrix}.
\end{align*}
We observe that by a direct calculation (recall that $\partial_{j} \Q$, $\Lambda\Q$ and $\Q_{1}$ are defined in Section~\ref{S:Spectral}.)
\[
E_{R}(T(\partial_{j} \Q))
	=0,
\quad
E_{R}(T(\Lambda\Q))
	=0,
\]
where $\Lambda Q:=2Q+x\cdot \nabla Q \in L^{2}(\R^{6})$,
and
\[
E_{I}(T(\Q_{1}))
	=0.
\]
We define $\F_{E}(\ve{u}, \ve{v})$ as the bilinear symmetric form
\begin{equation}\label{QuadraticNew}
\F_{E}(\ve{u}, \ve{v}) := \tfrac{1}{2} \<E_{R} \RE\ve{u}, \RE\ve{v}\> 
    + \tfrac{1}{2} \<E_{I} \IM\ve{u}, \IM\ve{v}\>,
\end{equation}
and we write $\F_{E}(\ve{u})$ to denote $\F_{E}(\ve{u}, \ve{u})$. Notice that $\F_{E}(T(\Q))<0$ for $\kk>0$.

The proof of the following lemma can be found in Appendix~\ref{S:A2}.

\begin{lemma}\label{SpecLL}
Fix $\kk>0$. Let $\sigma(\EE)$ be the spectrum of the operator $\EE$, defined on $(L^{2}(\R^{3}))^{4}$ with domain 
$(H^{2}(\R^{3}))^{4}$. Then $\EE$ has two simple eigenfunctions $e_{+} = (Y, Z)$ and $e_{-} = (\overline{Y}, \overline{Z})$ in $\Sch(\R^{3}) \times \Sch(\R^{3})$ with real eigenvalues
$\pm \lambda_{1}$ with $\lambda_{1} > 0$. Furthermore,
\[
\sigma(\EE) \cap \R = \left\{-\lambda_{1}, 0, \lambda_{1}\right\},
\quad
\text{and}\quad  
\sigma_{{\text ess}}(\EE) = \left\{i\xi: \xi \in \R\right\},
\]
where $\sigma_{\text{ess}}(\EE)$ denotes the essential spectrum of $\EE$.
\end{lemma}

\begin{remark}\label{PLf} 
A direct computation shows that for any $\ve{h}$, $\ve{u} \in \dot{H}^{1}$, we have
\begin{align*}
&\F_{E}(e_{\pm}) = \F_{E}(T(i\Q_{1})) = \F_{E}(T(\Lambda \Q)) = 0, \quad \F_{E}(T(\Q))  < 0,\\
&\F_{E}(\ve{h}, \ve{u}) = \F_{E}(\ve{u}, \ve{h}), \quad \F_{E}(\EE \ve{h}, \ve{u}) = -\F_{E}(\ve{h}, \EE \ve{u}),\\
&\F_{E}(\ve{h}, T(i\Q_{1})) =\F_{E}(\ve{h}, T(\Lambda \Q))=\F_{E}(\ve{h}, T(\partial_{j}\Q))=0,
\end{align*}
for $j = 1, \ldots, 4$.
\end{remark}

The proof of the following proposition can be found in Appendix~\ref{S:A2}.

\begin{proposition}\label{FCove}
Fix $\kk>0$. There exists $C > 0$ such that for every $\ve{h} \in G^{\bot}$, we have
\[
\F_{E}(\ve{h}) \geq C \|\ve{h}\|^{2}_{\dot{H}^{1}},
\]
where
\begin{align*}
\tilde{G}^{\bot} := \left\{\ve{h} \in \dot{H}^{1}, \F_{E}(\ve{h}, e_{\pm}) = (T(i\Q_{1}), \ve{h})_{\dot{H}^{1}} = (T(\Lambda \Q), \ve{h})_{\dot{H}^{1}} = 0, \right. \\
\left. (T(\partial_{j} \Q), \ve{h})_{\dot{H}^{1}} = 0 : j = 1, \ldots, 6 \right\}.
\end{align*}
\end{proposition}

\begin{remark}\label{PE1}
We have
\[
(e_{1}, T(\Q))_{H_{N}} \neq 0 \qtq{where} e_{1} = (\RE Y, \RE Z),
\]
and $(\cdot,\cdot )_{H_{N}}$ is the inner product associated with the norm $H_{N}^{\frac{1}{2}}$ (see \eqref{HNN}).
Indeed, suppose instead that $(e_{1}, T(\Q))_{H_{N}} = 0$. Notice that 
\[
\lambda_{1}\F_{E}(e_{\pm}, T(\Q)) = \pm\F_{E}(\EE e_{\pm}, T(\Q)) = \mp\F_{E}(e_{\pm}, \EE T(\Q))
= \tfrac{1}{2}\lambda_{1}(e_{1}, T(\Q))_{H_{N}} = 0.
\]
Here we have used that $E_{I}e_{2} = -\lambda_{1} e_{1}$, where $e_{2} = (\IM Y, \IM Z)$.
Since 
\[(T(i\Q_{1}), T(\Q))_{\dot{H}^{1}} = (T(\Lambda\Q), T(\Q))_{\dot{H}^{1}} = (T(\partial_{j}\Q), T(\Q))_{\dot{H}^{1}} = 0,\]
by Proposition~\ref{FCove}, we infer that $\F_{E}(T(\Q)) > 0$, which is a contradiction (cf. Remark~\ref{PLf}). 
\end{remark}

For $ I$ being a time interval, we will use the following notation:
\[
\begin{aligned}
&S(I) = L^{4}_{t}L^{4}_{x}(I \times \mathbb{R}^{6}) \times L^{4}_{t}L^{4}_{x}(I \times \mathbb{R}^{6}), \quad
Z(I) = L^{4}_{t}L^{\frac{12}{5}}_{x}(I \times \mathbb{R}^{6}) \times L^{4}_{t}L^{\frac{12}{5}}_{x}(I \times \mathbb{R}^{6}), \\
&N(I) = L^{2}_{t}L^{\frac{3}{2}}_{x}(I \times \mathbb{R}^{6}) \times L^{2}_{t}L^{\frac{3}{2}}_{x}(I \times \mathbb{R}^{6}), \quad
\text{L}^{p}(\mathbb{R}^{6}) = L^{p}(\mathbb{R}^{6}) \times L^{p}(\mathbb{R}^{6}),\\
&\text{S}:=\Sch(\R^{6})\times \Sch(\R^{6}) \qtq{Here, $\Sch(\R^{6})$ is the Schwartz space.}
\end{aligned}
\]

Note that
\begin{equation}\label{SI1}
\left\|\mathbf{f}\right\|_{S(I)} \lesssim \left\|\nabla \mathbf{f}\right\|_{Z(I)}.
\end{equation}

In the following lemmas, we obtain some estimates that will be useful throughout the section.
Recall that 
\[
B(\ve{h}) = B(h, g) = (\overline{h}Q + \sqrt{2\kk}Qg, \sqrt{2\kk}Qh) \qtq{and} N(\ve{h}) = (2\overline{h} g, h^{2}).
\]
\begin{lemma}[Linear estimates]\label{LL1}
Let $I$ be a finite interval of length $|I|$, $\ve{h}\in S(I)$ and $\nabla \ve{h}\in Z(I)$. There exists a constant $C$ independent of $I $ such that
\begin{equation}\label{Fi22}
\|\nabla B(\ve{h})\|_{N(I)}\leq |I|^{\frac{1}{4}}\|\nabla \ve{h}\|_{Z(I)}.
\end{equation}
Moreover, for $\ve{h}\in \text{L}^{3}$ we have
\begin{equation}\label{Fi1}
\|B(\ve{h})\|_{\text{L}^{\frac{3}{2}}}\leq C\|\ve{h}\|_{\text{L}^{3}}.
\end{equation}
\end{lemma}
\begin{proof}
The proof of \eqref{Fi1} is a direct consequence of the inequality 
\begin{equation}\label{Hnew}
\|fg\|_{L^{\frac{3}{2}}} \leq \|f\|_{L^{3}} \|g\|_{L^{3}}.
\end{equation}
On the other hand, using Hölder's inequality, we have that
\begin{equation}\label{H11}
\|{h}{g}\|_{L^{2}_{t}L^{\frac{3}{2}}_{x}} \leq \|{h}\|_{L^{4}_{t}L^{\frac{12}{5}}_{x}} \|{g}\|_{L^{4}_{t}L^{4}_{x}}.
\end{equation}
Thus, the inequality \eqref{Fi22} follows easily from the inequalities \eqref{SI1}, \eqref{H11}, and the fact that $|\partial_{\alpha}Q| \lesssim |Q|$ and $Q \in L^{4} \cap L^{\frac{12}{5}}$.
\end{proof}

\begin{lemma}[NonLinear estimates]\label{LN1}
Let $\ve{h}$ and $\ve{g}$ functions in $\text{L}^{3}(\mathbb{R}^{6})$. Then
\begin{equation}\label{Nol1}
\|N(\ve{h})-N(\ve{g})\|_{\text{L}^{\frac{3}{2}}}\leq C\|\ve{h}-\ve{g}\|_{\text{L}^{3}}\(\|\ve{h}\|_{\text{L}^{3}}+\|\ve{g}\|_{\text{L}^{3}}\).
\end{equation}
Let $I$ be a finite interval of length $|I|$, $\ve{h}$, $\ve{g} \in S(I)$ and $\nabla\ve{h}$, $ \nabla\ve{g}\in Z(I)$. There exists a constant $C$ independent of $I $ such that
\begin{equation}\label{Nol2}
\|\nabla N(\ve{h})-\nabla N(\ve{g})\|_{N(I)}\leq C\|\nabla \ve{h}-\nabla \ve{g}\|_{Z(I)}\(\|\nabla \ve{h}\|_{Z(I)}+\|\nabla \ve{g}\|_{Z(I)}\).
\end{equation}
\end{lemma}
\begin{proof}
From \eqref{Hnew}, we obtain \eqref{Nol1}. On the other hand, by combining \eqref{H11} and \eqref{SI1}, the inequality \eqref{Nol2} follows easily. We omit the details.
\end{proof}

We start by constructing approximate solutions to the linearized equation \eqref{De112}. The proof follows a similar approach to that of \cite[Proposition 6.3]{DuyMerle2009}, so it will suffice to outline the main steps of the argument.

\begin{proposition}\label{ApproxSo}
Let $ a \in \mathbb{R}$. There exists a sequence $\{g^{a}_{j}\}_{j \geq 1} $ in $ \text{S} $ such that the following holds: writing $ g^{a}_{1} = a e_{+} $ and
\[
U^{a}_{k}(t, x) :=\sum^{k}_{j=1} e^{-j \lambda_{1} t} g^{a}_{j}(x) \quad \text{for } k \geq 1,
\]
we have
\begin{equation}\label{Aproxlimi}
\epsilon_{k}:=\partial_{t} U^{a}_{k} + \mathcal{E} U^{a}_{k} - i N(U^{a}_{k}) =\mathcal{O}(e^{-(k+1) \lambda_{1} t}) \quad \text{in } \text{S} \quad \text{as } t \to \infty.
\end{equation}
Moreover, if $W^{a}_{k}=(f,h):=U^{a}_{k}+T(\Q)$, then 
\[
\epsilon_{k}:=i\partial_{t}W^{a}_{k}+(\Delta f, {\kk}\Delta g)+(\overline{f}g, g^{2})
=\mathcal{O}(e^{-(k+1) \lambda_{1} t}) \quad \text{in } \text{S} \quad \text{as } t \to \infty.
\]
\end{proposition}
\begin{proof}
The proof proceeds by induction. The proof proceeds by induction. Since $U^{a}_{1} = a e^{- \lambda_{1} t}e_{+}$, we see that
\[
\partial_{t} U^{a}_{1} + \mathcal{E} U^{a}_{1} - i N(U^{a}_{1}) = -i N(U^{a}_{1}) = \mathcal{O}(e^{-2\lambda_{1}t}).
\]
Thus, we obtain \eqref{Aproxlimi} for the case $k = 1$.

Now, let $k \geq 1$. Noting that $(k + 1)\lambda_{1}$  is not in the spectrum of $\mathcal{E}$ (see Lemma~\ref{SpecLL}), and following the same argument as in Step 2 of \cite[Proposition 6.3]{DuyMerle2009}, we derive the estimate \eqref{Aproxlimi}.
\end{proof}

Next, we recall a useful integral summation argument from \cite{DuyMerle2009}.
\begin{lemma}\label{SumsE}
Let $a_{0}>0$, $t_{0}>0$, $p\in[1, \infty)$, $E$ a normed vector space, and $f\in L_{\text{loc}}^{p}((t_{0}, \infty); E)$. Suppose that there exist $\tau_{0}>0$ and $C_{0}>0$ such that 
\[
\|f\|_{L^{p}(t, t+\tau_{0})}\leq C_{0}e^{-a_{0}t} \quad \text{for all $t\geq t_{0}$}.
\]
Then, for all $t\geq t_{0}$,
\[
\|f\|_{L^{p}(t, \infty)}\leq \frac{C_{0}e^{-a_{0}t}}{1-e^{-a_{0}\tau_{0}}}.
\]
\end{lemma}

We now construct true solutions to \eqref{NNLS} that remain close to the ground state $T(\Q)$ as $t\to\infty$.

\begin{proposition}\label{ContracA}
Let $a\in \R$. There exist $k_{0}>0$ and $t_{k}\geq 0$ such that for any $k\geq k_{0}$, there exists a solution $W^{a}$ to \eqref{NNLS} satisfying, for $t\geq t_k$,
\begin{equation}\label{Uniq}
\|\nabla W^{a}(t) - \nabla W_k^a(t)\|_{Z(t, +\infty)}
\leq e^{-(k+\frac{1}{2})\lambda_{1}t}.
\end{equation}
Furthermore, $W^{a}$ is the unique solution to equation \eqref{NNLS} that satisfies \eqref{Uniq} for large $t$. Finally, $W^{a}$ is independent of $k$ and satisfies, for large $t$,
\begin{equation}\label{UniqVec}
\|W^{a}(t)-T(\Q)-ae^{-\lambda_{1} t}e_{+}\|_{\dot{H}^{1}}\leq e^{-\frac{3}{2}\lambda_{1}t}.
\end{equation}
\end{proposition}
\begin{proof}
The argument is very similar to the one provided in \cite[Proposition 6.3]{DuyMerle2009}.

Notice that the function $W^a$ is a solution of \eqref{NNLS} if and only if $\ve{w}^a := W^a - T(\Q)$ is a solution of
\[
\partial_t \ve{w}^a + \EE \ve{w}^a = iN(\ve{w}^a).
\]
Moreover, by \eqref{Aproxlimi}, $\ve{v}_k^a := W_k^a - T(\Q)$ satisfies the identity
\[
\partial_t \ve{v}_k^a + \EE \ve{v}_k^a - iN(\ve{v}_k^a) = \varepsilon_k .
\]
Thus, $W^a$ satisfies \eqref{NNLS} if and only if $\ve{h} := W^a - W_k^a = \ve{w}^a - \ve{v}_k^a$ satisfies
\[
\partial_t \ve{h} + \EE \ve{h} = iN(\ve{v}_k^a + \ve{h}) - iN(\ve{v}_k^a) - \varepsilon_k.
\]
This can be rewritten as (with $\ve{h}:=(h,g)$)
\[
i \partial_t \ve{h} + (\Delta h, {\kk}\Delta g) = -B(\ve{h}) - N(\ve{v}_{k}+ \ve{h}) + N(\ve{v}_{k}) - i \varepsilon_k.
\]

We construct the solution to \eqref{NNLS} via a fixed point argument. Consider the operator
\begin{equation*}
[\text{M}_{k}(\ve{h}) ](t):=-\int^{\infty}_{t}S(t-s)[-iB(\ve{h}(s))-iN(\ve{v}_{k}(s)+\ve{h}(s))+iN(\ve{v}_{k}(s))+\epsilon_{k}(s)]\,ds,
\end{equation*}
where
\[
S(t)=\begin{bmatrix}
e^{it\Delta} & 0 \\
0 & e^{i\kk t\Delta}
\end{bmatrix},
\]
and the spaces
\begin{align*}
E^{k}_{Z}&:=\left\{ \ve{h} \in S(t_k, +\infty), \nabla \ve{h} \in Z(t_k, +\infty); \|h\|_{E_l^k} := \sup_{t \geq t_k} e^{(k + \frac{1}{2})\lambda_{1} t} \| \nabla \ve{h}\|_{Z(t, +\infty)} < \infty \right\},\\
L^{k}_{Z}&:=\left\{\ve{h} \in E^{k}_{Z}, \|\ve{h}\|_{E^{k}_{Z}}\leq 1\right\}.
\end{align*}
Note that the space $E_Z^k$ is a Banach space. We need the following:
\begin{claim}\label{Est34}
There exists $k_{0}>0$ such that for $k\geq k_{0}$, we have
\begin{align}\label{Cla11}
\| \nabla B (\ve{h})\|_{N(t, \infty)}&\leq \tfrac{1}{4 C^{\ast}}e^{-(k+\frac{1}{2})\lambda_{1}t}\|\ve{h}\|_{E_{Z}^{k}},\\ 
\label{Cla22}
\|\nabla (N(\ve{v}_{k}+\ve{g})-N(\ve{v}_{k}+\ve{h}))\|_{N(t, \infty)}&\leq C_{k}e^{-(k+\frac{3}{2})\lambda_{1}t}\|\ve{g}-\ve{h}\|_{E_{Z}^{k}},\\
\label{Cla33}
\|\epsilon_{k}\|_{N(t, \infty)}&\leq C_{k}e^{-(k+1)\lambda_{1}t},
\end{align}
for $t\geq t_{k}$, where the constant $C_{k}$ depends only on $k$.
\end{claim}
\begin{proof}[{Proof of Claim~\ref{Est34}}]
We first estimate \eqref{Cla11}. Let $\tau_{0}>0$. From \eqref{Fi22}, we have
\[
\|\nabla B(\ve{h})\|_{N(t, t+\tau_{0})}\leq C_{1}\tau_{0}^{\frac{1}{4}}e^{-(k+\frac{1}{2})\lambda_{1}t}\|\ve{h}\|_{E_{Z}^{k}}.
\]
Then obtain \eqref{Cla11} by Lemma~\ref{SumsE} for $k\geq k_{0}$ by choosing $\tau_{0}$ and $k_{0}$ appropriately.

Now we show \eqref{Cla22}. Notice that by definition $\|\ve{v}_{k}\|_{Z(t, t+1)}\leq C_{k}e^{-\lambda_{1}t}$ 
(cf. Proposition~\ref{ContracA}). Moreover, by estimate \eqref{Nol2}, we deduce (we set $I:=[t,t+1]$)
\begin{align*}
&\|\nabla (N(\ve{v}_{k}+\ve{g}) - N(\ve{v}_{k}+\ve{h}))\|_{N(I)} \\
&\quad \leq C_{1,2} \|\nabla \ve{h} - \nabla \ve{g}\|_{Z(I)} \left( \|\nabla \ve{h}\|_{Z(I)} + \|\nabla \ve{g}\|_{Z(I)} + \|\nabla \ve{v}_{k}\|_{Z(I)} \right) \\
&\quad \leq C_{k,2} e^{-\lambda_{1}t} \|\nabla \ve{h} - \nabla \ve{g}\|_{Z(I)} \\
&\quad \leq C_{k,2} e^{-(k+\frac{3}{2})\lambda_{1}t} \|\ve{h} - \ve{g}\|_{E_{Z}^{k}}.
\end{align*}
where the constant $C_{k,2}$ depends only on $k$. Therefore, Lemma~\ref{SumsE} implies \eqref{Cla22}. 

Finally, \eqref{Cla33} is a direct consequence of \eqref{Aproxlimi}.
\end{proof}

With Claim~\ref{Est34}, the remainder of the proof follows the same  as in Step 2 of \cite[Proposition 6.3]{DuyMerle2009}. This completes the proof.
\end{proof}

\subsection{Construction of special solutions}
\begin{proof}[Proof of Theorem~\ref{Gcharc}]
As the norms $H^{\frac{1}{2}}_{N}(\cdot)$ (cf. \eqref{HNN}) and $\|\cdot\|_{\dot{H}^{1}}$ are equivalent,  by Proposition~\ref{ContracA}, we have
\[
H_{N}(W^{a}(t))=H_{N}(T(\Q))+2ae^{-\lambda_{1}t}(e_{1}, T(\Q))_{H_{N}}+O\left(e^{-\frac{3}{2}\lambda_{1}t}\right)\quad \mbox{as $t\to+\infty$}.
\]
From Remark~\ref{PE1}, we may assume that $(e_{1}, T(\Q))_{H_{N}}>0$. Therefore, we have that $H_{N}(W^{a}(t))-H_{N}(T(\Q))$ has the same sign as $a$ for large times. In particular, by variational characterization (cf. Proposition~\ref{TheGN}), we infer that $H_{N}(W^{a}(t_{0}))-H_{N}(T(\Q))$ has the same sign as $a$.
We define 
\[
A^{+}(t,x)=W^{+1}(t+t_{0},x),\quad A^{-}(t,x)=W^{-1}(t+t_{0},x),
\] 
for $t_{0}$ sufficiently large. Then, we obtain two radial solutions $A^{\pm}(t,x)$ of \eqref{NNLS} that satisfy

\[
H_{N}(A^{-}(0))<H_{N}(T(\Q)) \qtq{and} H_{N}(A^{+}(0))>H_{N}(T(\Q)), 
\]
and
\[
H_{N}(A^{\pm}(t)-T(\Q))\leq Ce^{-\lambda_{1}t}\qtq{for $t\geq 0$.}
\]
Next, we define $\G^{+}=T^{-1}(A^{+})$ and $\G^{-}=T^{-1}(A^{-})$ (recall $T$ in \eqref{TRA}). Then, $\G^{+}$ and $\G^{-}$ are solutions of \eqref{SNLS}, and since 
$H(T^{-1}(\ve{u}))=2H_{N}(\ve{u})$, we see that
\[
\|\G^{\pm}(t)-\Q\|^{2}_{\dot{H}^{1}}\lesssim H(\G^{\pm}(t)-\Q)\leq Ce^{-\lambda_{1}t} \quad \text{for $t\geq 0$},
\]
and
\[
H(\G^{-}(0))<H(\Q) \qtq{and} H(\G^{+}(0))>H(\Q). 
\]
Note also that, by Corollary~\ref{ClassC}, the solution $\G^{-}$ is defined for all $\R$ and scatters as $t \to-\infty$. 
Finally, the same argument given in \cite[p.~38]{DuyMerle2009} shows that $\mathcal{G}^{+}(0) \in L^{2}$. Thus, Proposition~\ref{SupercriQ} implies that the negative existence time of $\mathcal{G}^{+}$ is finite.

This concludes the proof of the theorem.
\end{proof}

\section{A Uniqueness Result}\label{S:uniq}

Our first main goal in this section is to establish the following:

\begin{proposition}\label{UniqueU}  
If $\ve{u}$ is a radial solution to \eqref{NNLS} satisfying
\begin{equation}\label{UniqCon}
\|\ve{u}(t)-T(\Q)\|_{\dot{H}^{1}}\leq Ce^{-ct} \qtq{for $t\geq 0$,}
\end{equation} 
for some $C$, $c>0$, then there exists a unique $a \in \R$ such that $\ve{u}=W^{a}$, where $W^{a}$ is the solution of \eqref{NNLS} given in Proposition~\ref{ContracA}.
\end{proposition}

\begin{corollary}\label{coroCla}
Let $a\neq0$. Then there exists  $T_{a}\in \R$ such that 
\begin{equation}\label{C1s}
\begin{cases} 
W^{a}=W^{+1}(t+T_{a}) & \text{if } a>0, \\
W^{a}=W^{-1}(t+T_{a})  & \text{if } a<0.
\end{cases}
\end{equation}
\end{corollary}

We begin with a lemma.

\begin{lemma}\label{BoundGra11}
Let $\ve{v}$ be a solution of \eqref{De112} with 
\begin{equation}\label{ExpH}
\|\ve{v}(t)\|_{{ \dot{H}}^{1}}\leq Ce^{-c_{0}t}
\end{equation}
for some constants $C>0$ and $c_{0}>0$. Then for any admissible pair $(q, r)$ we have for $t$ large
\begin{equation}\label{BoundCo}
\|\ve{v}\|_{S(t, +\infty)}
+\|\nabla \ve{v}\|_{\text{L}^{p}(t, +\infty; L^{q})}
\leq Ce^{-c_{0}t}.
\end{equation}
\end{lemma}

\begin{proof}
Using Lemmas~\ref{SumsE}, \ref{LL1}, and \ref{LN1}, the proof follows the same lines as \cite[Lemma 5.7]{DuyMerle2009}.
\end{proof}

For what follows, we introduce the linearized equation
\begin{equation}\label{CondiExp}
\partial_{t}\ve{v}+\EE \ve{v}=g, \quad (t,x)\in [0, \infty)\times \R^{6},
\end{equation}
where $\ve{v}$ and $\epsilon$ satisfy, for all $t\geq 0$,
\begin{align} \label{CondiExp22}
& \|\ve{v}(t)\|_{{ \dot{H}}^{1}}\leq Ce^{-c_{1}t},\\ \label{CondiExp33}
&\|\nabla g\|_{N(t, +\infty)}+\|g\|_{L_{x}^{\frac{3}{2}}}\leq Ce^{-c_{2}t},
\end{align}
with $0 < c_{1} < c_{2}$.

Notice that by Strichartz estimates and Lemma~\ref{SumsE}, we can obtain the following result.

\begin{lemma}\label{AxuST11}
Under the assumptions \eqref{CondiExp}, \eqref{CondiExp22}, and \eqref{CondiExp33} with $0 < c_{1} < c_{2}$, we have, for any admissible pair $(q, r)$,
\begin{equation}\label{NewSt11}
\|\ve{v}\|_{\text{L}^{p}(t, +\infty; L^{q})}
\leq Ce^{-c_{1}t}.
\end{equation}
\end{lemma}

In what follows, we will use the following notation: for a given $c > 0$, we denote by $c^{-}$ a positive number that is arbitrarily close to $c$ and satisfies $0 < c^{-} < c$. Additionally, recall that $\lambda_1 > 0$ represents the eigenvalue of the linearized operator $\EE$, as defined in Lemma~\ref{SpecLL}.

\begin{proposition}\label{AxuST}
Consider $\ve{v}$ and $\epsilon$ satisfying  \eqref{CondiExp}, \eqref{CondiExp22} and \eqref{CondiExp33}. Then we have:
\begin{enumerate}[label=\rm{(\roman*)}]
\item If $\lambda_{1}\notin [c_{1}, c_{2})$, then 
\begin{equation}\label{BoundHsec}
 \|\ve{v}(t)\|_{{ \dot{H}}_{1}}\leq Ce^{-c^{-}_{2}t}.
\end{equation}
\item If $\lambda_{1}\in [c_{1}, c_{2})$, then there exists $a\in \R$ such that
\begin{equation}\label{BoundHsec22}
 \|\ve{v}(t)-ae^{-\lambda_{1}t}e_{+}\|_{{ \dot{H}}^{1}}
\leq Ce^{-c^{-}_{2}t}.
\end{equation}
\end{enumerate}
\end{proposition}
\begin{proof}
We closely follow the argument in \cite[Proposition 5.9]{DuyMerle2009} and \cite[Proposition 7.2]{MIWUGU2015}.
Let 
\[
Y^{\bot}:= \left\{\ve{h} \in \dot{H}^{1}, \F_{E}(\ve{h}, e_{\pm}) = (T(i\Q_{1}), \ve{h})_{\dot{H}^{1}} = 
 (T(\Lambda \Q), \ve{h})_{\dot{H}^{1}} =0\right\}.
\]

We decompose $\ve{v}$ as
\begin{equation}\label{DecompV}
\ve{v}(t)=\alpha_{+}(t)e_{+}+\alpha_{-}(t)e_{-}+\beta(t)T(i\Q)+\gamma(t)T(\Lambda \Q)+v^{\bot}(t),
\end{equation}
where $v^{\bot}(t)\in Y^{\bot}\cap \dot{H}^{1}_{rad}$.

Notice that by Remark~\ref{Nzero} we have that $\F_{E}(e_{+}, e_{-})\neq 0$, thus, we can normalize the eigenfunctions $e_{\pm}$ such that $\F_{E}(e_{+}, e_{-})=1$. Moreover, from the definition of $Y^{\bot}$ and Remark~\ref{PLf}, we see that
\begin{align*}
&	\alpha_{+}(t)=\F_{E}(\ve{v}(t), e_{-}), \quad	\alpha_{-}(t)=\F_{E}(\ve{v}(t), e_{+}),\\
& \beta(t)=\frac{1}{\|T(\Q)\|_{\dot{H}^{1}}}(\ve{v}(t)-\alpha_{+}(t)e_{+}-\alpha_{-}(t)e_{-}, T(i\Q))_{\dot{H}^{1}}\\
& \gamma(t)=\frac{1}{\|T(\Lambda\Q)\|_{\dot{H}^{1}}}(\ve{v}(t)-\alpha_{+}(t)e_{+}-\alpha_{-}(t)e_{-}, T(\Lambda\Q))_{\dot{H}^{1}}.
\end{align*}

\textbf{Step 1. Decay estimates.} From \eqref{CondiExp33}, \eqref{CondiExp22} and following the same argument developed in \cite[Proposition 7.2, Step 2]{MIWUGU2015}, one can show that
\begin{align}\label{Exp11alfa00}
&|\alpha^{\prime}_{-}(t)|\leq C e^{-c_{2}t},\\ \label{Exp11alfa}
&|\alpha^{\prime}_{+}(t)|\leq C e^{-c_{2}t}	\qtq{if $\lambda_{1}\leq c_{1}$ or $c_{2}\leq \lambda_{1}$},\\\label{Exp11alfalamda}
&|\alpha_{+}(t)-ae^{-\lambda_{1}t}|\leq e^{-c_{2}t} \quad \text{if $c_{1}\leq \lambda_{1}<c_{2}$},
\end{align}
where
\begin{align}\label{L:alfa}
a:=\lim_{t\to \infty}e^{\lambda_{1}t}\alpha_{+}(t).
\end{align}
Indeed, note that for some time interval $I$ with $|I|< \infty$, we have
\begin{equation*}
\begin{aligned}
&\int_{I}\left|\int_{\R^{6}} \nabla f(t)\nabla g(t) dx\right|\lesssim\|\nabla {f}\|_{L^{2}(I :L^{\frac{3}{2}})} \|\nabla {g}\|_{{L}^{2}(I: {L}^{3})}\\
&\int_{\R^{6}}  \left| f\,g\, Q\right|\, dx\lesssim \|f\|_{{L}_{x}^{\frac{3}{2}}}\|g\|_{{L}_{x}^{3}}\|Q\|_{{L}^{\infty}}.
\end{aligned}
\end{equation*}
Combining these inequalities, by the definition of $\F_{E}$, we see that
\begin{equation}\label{L1BoundF}
\begin{aligned}
\int_{I}|\F_{E}(\ve{f}(t), \ve{h}(t))|dt &\lesssim
\|\nabla \ve{f}\|_{N(I)} \|\nabla \ve{h}\|_{\text{L}^{2}(I: \text{L}^{3})}\\
&+|I| \| \ve{f}\|_{\text{L}^{\infty}(I: \text{L}^{\frac{3}{2}})}\| \ve{h}\|_{\text{L}^{\infty}(I: \text{L}^{3})}.
\end{aligned}
\end{equation}
Thus, by \eqref{CondiExp33} and inequality \eqref{L1BoundF}, we get
\[
\int^{t+1}_{t}|e^{-\lambda_{1}s}\F_{E}(g(s), e_{+})|ds\leq Ce^{-(\lambda_{1}+c_{2})t}.
\]
In this case, Lemma~\ref{SumsE} shows that
\[
\int^{\infty}_{t}|e^{-\lambda_{1}s}\F_{E}(g(s), e_{+})|ds\leq Ce^{-(\lambda_{1}+c_{2})t}.
\]

From the above inequalities, and following exactly the same argument as in \cite[Proposition 7.2]{MIWUGU2015}, we obtain the estimates \eqref{Exp11alfa00}, \eqref{Exp11alfa}, and \eqref{Exp11alfalamda}.

\textbf{Step 2. Proof in the case when either  $\lambda_1\geq c_2$, or $\lambda_1<c_2$ and $a=0$.}
Note that from the estimates in the previous step, we have
\begin{equation}\label{Al12}
|\alpha_{+}(t)|+|\alpha_{-}(t)| \leq Ce^{-c_{2}t}.
\end{equation}
On the other hand, note that  
\begin{equation}\label{DecomF}
\frac{d}{dt}\F_{E}(\ve{v}(t))=2\F_{E}(\partial_{t} \ve{v}(t), \ve{v}(t))=-2\F(\L \ve{v},\ve{v})+2\F(g, \ve{v})=2\F(g, \ve{v}).
\end{equation}
Moreover, by \eqref{CondiExp33} and inequality \eqref{L1BoundF}, we get
\[
\int^{t+1}_{t}|\F_{E}(g(s), \ve{v}(s))|ds \leq Ce^{-(c_{1}+c_{2})t}.
\]
Then Lemma~\ref{SumsE} implies that
\[
\int^{\infty}_{t}|\F_{E}(g(s), \ve{v}(s))|ds \leq Ce^{-(\lambda_{1}+c_{2})t}.
\]
Since $|\F_{E}(\ve{v}(t))| \lesssim \|\ve{v}(t)\|^{2}_{H^{1}} \to 0$ as $t \to \infty$ (cf. \eqref{CondiExp22}), by \eqref{DecomF} and the inequality above, we have
\[
|\F_{E}(\ve{v}(t))| \leq \int^{\infty}_{t}|\F_{E}(g, \ve{v}(t))|dt \leq Ce^{-(c_{1}+c_{2})t}.
\]
Since $\F_{E}(e_{+}, e_{-})=1$ and $\F_{E}(e_{+})=\F_{E}(e_{-})=0$, we obtain
\[
\F_{E}(\ve{v})=\F(v^{\bot})+2\alpha_{+}\alpha_{-}.
\]
Thus, Proposition~\ref{FCove} and \eqref{Al12} imply that 
\begin{equation}\label{EstimaV}
\|v^{\bot}(t)\|_{\dot{H}^{1}} \lesssim \sqrt{|\F_{E}(v^{\bot})|} \leq Ce^{-\frac{(c_{1}+c_{2})}{2}t}.
\end{equation}
In addition, we have that
\begin{equation}\label{EstimaV123}
(T(i\Q), \EE v^{\bot})_{\dot{H}^{1}}=(\EE^{\ast}T(i\Delta\Q),  v^{\bot})_{{L}^{2}}
\lesssim \|\EE^{\ast}T(i\Delta\Q)\|_{L^{\frac{3}{2}}}\|v^{\bot}\|_{\dot{H}^{1}}
\lesssim e^{-\frac{(c_{1}+c_{2})}{2}t}.
\end{equation}
Here, we used the fact that $\EE^{\ast}T(i\Delta\Q)=E_{R}T(\Delta\Q) \in L^{\frac{3}{2}}$.
Combining the inequalities \eqref{EstimaV} and \eqref{EstimaV123}, and following the same lines as in Step 3 of \cite[Proposition 7.2]{MIWUGU2015}, we obtain the proof of the proposition in the case when either $\lambda_1 \geq c_2$, or $\lambda_1 < c_2$ and $a=0$.

\textbf{Step 3. Proof in the remaining cases.} If $\lambda_1 < c_1$, then we have $\lambda_1 < c_2$. Moreover, in this case, by \eqref{L:alfa}, we infer that $a = 0$, and hence we obtain the estimate in (ii) with $a = 0$ using Step 2. Thus, it suffices to consider $c_1 \leq \lambda_1 < c_2$ and $a \neq 0$. The proof in this case is the same as the one given in Step 4 of \cite[Proposition 7.2]{MIWUGU2015}.
This completes the proof of the proposition.

\end{proof}

\begin{proof}[Proof of Proposition~\ref{UniqueU}]
With Propositions~\ref{AxuST} and \ref{ContracA}, and Lemmas~\ref{BoundGra11}, \ref{LN1}, \ref{AxuST11} and \ref{SumsE} at hand, the proof follows along the same lines as in \cite[Lemma 6.5]{DuyMerle2009}. We omit the details.
\end{proof}

\begin{proof}[Proof of Corollary~\ref{coroCla}]
Let $a\neq0$ and choose $T_{a}$ so that $|a|e^{-\lambda_{1}T_{a}}=1$. By estimate \eqref{UniqVec} we obtain
\begin{equation}\label{W1}
\|W^{a}(t+T_{a})-T(\Q)\mp e^{-\lambda_{1} t}e_{+}\|_{\dot{H}^{1}}
\leq e^{-\frac{3}{2}\lambda_{1}t}.
\end{equation}

On the other hand, $W^{a}(t+T_{a})$ satisfies the assumption of Proposition~\ref{UniqueU}. Therefore, there exists $\tilde{a}$
such that 
\[W^{a}(\cdot+T_{a})=W^{\tilde{a}}.\]
From \eqref{W1} and Proposition~\ref{ContracA} we obtain that $\tilde{a}=1$ if $a>0$, and $\tilde{a}=-1$ if $a<0$, which implies \eqref{C1s}. 
\end{proof}

\section{Proof of the main result}\label{S:proof}

\begin{proof}[Proof of Theorem~\ref{TH22}]

(i) Let $\ve{u}$ be a radial solution to \eqref{SNLS} such that
\begin{equation}\label{Hdn}
E(\ve{u}_0) = E(\Q), \quad H(\ve{u}_0) < H(\Q).
\end{equation}
By Lemma~\ref{GlobalW}, the solution $\ve{u}$ is global. Suppose  that $\ve{u}$ does not scatter, i.e., $\|\ve{u}\|_{\text{L}^{4}_{t, x}(\R \times \R^6)} = \infty$. 

If necessary, replace $\ve{u}(t)$ with $\overline{\ve{u}}(-t)$. Then, by Proposition~\ref{CompacDeca} and Corollary~\ref{ClassC}, there exist $\theta_0 \in \mathbb{R}$, $\mu_0 > 0$, and constants $c, C > 0$ such that (recall that the norms $[H(\cdot)]^{\frac{1}{2}}$,  $[H_{N}(\cdot)]^{\frac{1}{2}}$ and $\dot{H}^1$ are equivalents)
\[
H(\ve{u}_{[\theta_0, \mu_0]}(t) - \Q)\|\lesssim \|\ve{u}_{[\theta_0, \mu_0]}(t) - \Q\|_{\dot{H}^1} \leq Ce^{-ct} \qtq{for $t\geq 0$.}
\]
In particular, we have
\begin{equation}\label{DNb}
\|T(\ve{u}_{[\theta_0, \mu_0]})(t) - T(\Q)\|_{\dot{H}^1} \leq Ce^{-ct} \qtq{for $t\geq 0$.}
\end{equation}

Since $T(\ve{u}_{[\theta_0, \mu_0]})$ is a radial solution to \eqref{NNLS} and satisfies \eqref{DNb}, we observe that $T(\ve{u}_{[\theta_0, \mu_0]})$ fulfills the assumptions of Proposition~\ref{UniqueU}. Moreover, \eqref{Hdn} implies that $H_N(T(\ve{u}_{[\theta_0, \mu_0]})) < H_N(T(\Q))$. Therefore, Corollary~\ref{coroCla} implies the existence of $a < 0$ and $T_a$ such that
\[
T(\ve{u}_{[\theta_0, \mu_0]})(t) = W^{-1}(t + T_a).
\]
Thus, $\ve{u}_{[\theta_0, \mu_0]}(t) = T^{-1}(W^{-1}(t + T_a)) = \G^{-}(t+t_{0})$ for some $t_{0}\in \R$, which completes the proof of part (i).

(ii) If $E(\ve{u}_0) = E(\Q)$ and $H(\ve{u}_0) = H(\Q)$, then by the variational characterization provided in Proposition~\ref{TheGN}, we conclude that $\ve{u}_0 = \Q$ up to the symmetries of the equation.

Now, we prove part (iii). Let $\ve{u}$ be a radial solution to \eqref{NNLS} defined on $[0, +\infty)$ (if necessary, replace $\ve{u}(t)$ with $\overline{\ve{u}}(-t)$) such that
\[
E(\ve{u}) = E(\Q), \quad H(\ve{u}_0) > H(\Q), \quad \text{and} \quad \ve{u}_0 \in {L}^2.
\]
By Proposition~\ref{SupercriQ}, there exist $\theta_0 \in \mathbb{R}$, $\mu_0 > 0$, and constants $c, C > 0$ such that
\[
\|\ve{u}_{[\theta_0, \mu_0]}(t) - \Q\|_{\dot{H}^1} \leq Ce^{-ct} \qtq{for $t\geq 0$.}
\]
In a manner similar to the proof of (i), we conclude that
\begin{equation}\label{DNWu}
\|T(\ve{u}_{[\theta_0, \mu_0]})(t) - T(\Q)\|_{\dot{H}^1} \leq Ce^{-ct} \qtq{for $t\geq 0$.}
\end{equation}
Thus, since $H_N(T(\ve{u}_{[\theta_0, \mu_0]})) > H_N(T(\Q))$, by Corollary~\ref{coroCla} there exist $a > 0$ and $T_a$ such that
\[
T(\ve{u}_{[\theta_0, \mu_0]})(t) = W^{+1}(t + T_a).
\]
This implies that $\ve{u}_{[\theta_0, \mu_0]}(t) = T^{-1}(W^{+1}(t + T_a)) = \G^{+}(t+t_{0})$ for some $t_{0}\in \R$, which completes the proof of part (iii).

This concludes the proof of the theorem.

\end{proof}

\appendix
\section{Coercivity}\label{S:A}

In this appendix, we will study some coercivity properties of the operators $L_{R}$ and $L_{I}$, which are defined in Section~\ref{S:Spectral}.

For $\gamma\in \R$ we define $L_{\gamma}v=-\Delta v-\gamma Qv$ for $v\in \dot{H}^{1}(\R^{6})$.  

\begin{lemma}\label{Apx22}
Fix $\kk>0$. There exists $C>0$ such that for every $\ve{v}\in \dot{H}^{1}(\R^{6}: \R)\times \dot{H}^{1}(\R^{6}: \R)$ satisfying
\begin{align}\label{Ort223}
	(\ve{v}, \Xi_{1} )_{\dot{H}^{1}}=0, \qtq{where $\Xi_{1}=(Q ,0)$,}
\end{align}
 then we have
\[
\<L_{I}\ve{v},\ve{v}
	\>
	\geq C \|\ve{v}\|^{2}_{\dot{H}^{1}}.
\]
\end{lemma}
\begin{proof}
We observe (we set $\ve{v}=(u,v)$)
\[
\<L_{I}\ve{v},\ve{v}\>=\|\nabla u\|^{2}_{L^{2}}+\tfrac{\kk}{2}\|\nabla v\|^{2}_{L^{2}}+\int_{\R^{6}}Q|u|^{2}\,dx
-2\sqrt{\kk}\int_{\R^{6}}Q\,u\, v\,dx.
\]
By Young's inequality  we see that
\begin{equation}\label{AYIne}
\sqrt{\kk}\int_{\R^{6}}Q\,u\, v\,dx\leq \tfrac{1}{2}\int_{\R^{6}}2Q|u|^{2}\,dx+\tfrac{1}{2}\int_{\R^{6}}\tfrac{\kk}{2}Q|v|^{2}\,dx.
\end{equation}
Combining the inequalities above, we infer that
\begin{equation}\label{DeLi}
\<L_{I}\ve{v},\ve{v}\>\geq \<L_{1}u,u\>+\tfrac{\kk}{2}\<L_{1}{v},{v}\>.
\end{equation}
In particular, note that $\<L_{I}\ve{v},\ve{v}\> \geq 0$. Now, assume that $(u,v)\neq (0,0)$. First, if $u\neq 0$ and $v\in  \dot{H}^{1}(\R^{6}: \R)$, then by \eqref{Ort223} and \eqref{DeLi} we have
\[
\<L_{I}\ve{v},\ve{v}\>\geq \tfrac{\kk}{2}\<L_{1}{v},{v}\>+\<L_{1}u,u\>\geq C_{1}\|u\|^{2}_{\dot{H}^{1}}>0.
\]
On the other hand, if $u=0$ and $v\neq0$, then
\[
\<L_{I}\ve{v},\ve{v}\>=\tfrac{\kk}{2}\|v\|^{2}_{\dot{H}^{1}}>0.
\]
In any case, if $\ve{v}\neq (0,0)$ and satisfies \eqref{Ort223}, we obtain that $\<L_{I}\ve{v},\ve{v}\>>0$. 
As the quadratic form $\<L_{I}\ve{v},\ve{v}\>$ is a compact perturbation of $\|\cdot\|^{2}_{\dot{H}^{1}}+\tfrac{k}{2}\|\cdot\|^{2}_{\dot{H}^{1}}$, a standard argument shows the conclusion of the lemma.
\end{proof}

Next, we set 
\begin{equation}\label{piY}
\Pi_{1}=\big(\sqrt{\tfrac{2}{3}}Q, \tfrac{1}{\sqrt{k}}\sqrt{\tfrac{2}{3}}Q\big), \quad
\Pi_{2}=\big(\sqrt{\tfrac{2}{3}}\Lambda Q, \tfrac{1}{\sqrt{k}}\sqrt{\tfrac{2}{3}}\Lambda Q\big), \quad
\Psi_{j}=\big(\sqrt{\tfrac{2}{3}}\partial_{j}Q, \tfrac{1}{\sqrt{k}}\sqrt{\tfrac{2}{3}}\partial_{j}Q\big)
\end{equation}
for $1\leq j\leq 6$.

\begin{lemma}\label{Apx11}
Fix $\kk>0$. There exists $C>0$ such that for every $\ve{v}\in \dot{H}^{1}(\R^{6}: \R)\times \dot{H}^{1}(\R^{6}: \R)$ satisfying
\begin{align}\label{Ort11}
	(\ve{v}, \Pi_{1} )_{\dot{H}^{1}}=(\ve{v}, \Pi_{2} )_{\dot{H}^{1}}=(\ve{v}, \Psi_{j} )_{\dot{H}^{1}}=0
\end{align}
for $1\leq j\leq 6$, then we have
\[
\<L_{R}\ve{v},\ve{v}
	\>
	\geq C \|\ve{v}\|^{2}_{\dot{H}^{1}}.
\]
\end{lemma}
\begin{proof}
To obtain the result, it is not possible to use the same argument as in Lemma~\ref{Apx22}. For this reason, we will use a different approach. Consider
\[
\A=
\begin{bmatrix}
-\Delta & 0 \\
0 & -\Delta 
\end{bmatrix}
-
\begin{bmatrix}
Q & \sqrt{2}Q \\
\sqrt{2}Q & 0 
\end{bmatrix}.
\]

 Then $\A$ is diagonalized as follows:
\[
\A=P\begin{pmatrix}
L_{2} & 0 \\
 0 & L_{-1} 
\end{pmatrix}
P^{\ast},
\qtq{where}
P = \begin{bmatrix}
\frac{\sqrt{2}}{\sqrt{3}} & \frac{1}{\sqrt{3}} \\
\frac{1}{\sqrt{3}} & -\frac{\sqrt{2}}{\sqrt{3}}
\end{bmatrix}.
\]
 Notice that 
\[\<\A\ve{v},\ve{v}\>=\<L_{2}w_{1}, w_{1} \>+\< L_{-1} w_{2}, w_{2} \>,\]
 where $\ve{w}=P^{\ast}\ve{v}$. Next, we define the transformation 
$\Gamma(\ve{v}) = \Gamma(u, v) := \left(u, \frac{\sqrt{2}}{\sqrt{k}}v\right)$. A simple calculation shows that
\[
\<L_{R}\ve{v}, \ve{v}\> = \<L_{R}\Gamma(\Gamma^{-1}\ve{v}), \Gamma(\Gamma^{-1}\ve{v})\> = \<\A\Gamma^{-1}\ve{v}, \Gamma^{-1}\ve{v}\> = \<L_{2}\tilde{w}_{1}, \tilde{w}_{1} \> + \< L_{-1} \tilde{w}_{2}, \tilde{w}_{2} \>,
\]
where $(\tilde{w}_{1}, \tilde{w}_{2}) = P^{\ast}\Gamma^{-1}\ve{v}$. Since (see \eqref{Ort11})
\[
(\tilde{w}_{1}, Q)_{\dot{H}^{1}}=(\tilde{w}_{1}, \Lambda Q)_{\dot{H}^{1}}=(\tilde{w}_{1}, \partial_{j}Q)_{\dot{H}^{1}}=0
\]
for $1\leq j\leq 6$,  Step 1 in \cite[Lemma 3.5]{CamposFarahRoudenko} implies  that there exists a constant $C_{1}$ such that
\[
\<L_{2}\tilde{w}_{1}, \tilde{w}_{1} \> \geq C_{1}\|\tilde{w}_{1}\|^{2}_{\dot{H}^{1}}.
\]
Therefore, if $\ve{v}\neq 0$ and satisfies \eqref{Ort11}, we have
\[
\<L_{R}\ve{v}, \ve{v}\> \gtrsim \|P^{\ast}\Gamma^{-1}\ve{v}\|^{2}_{\dot{H}^{1}} = \|\Gamma^{-1}\ve{v}\|^{2}_{\dot{H}^{1}} \gtrsim \|\ve{v}\|^{2}_{\dot{H}^{1}}.
\]
This completes the proof of the lemma.

\end{proof}

\section{Spectrum of the linearized operator}\label{S:A2}

Throughout this section, we assume that $\kk>0$. The main goal of this appendix is to provide the proof of Proposition~\ref{FCove} and Lemma~\ref{SpecLL}. To this end, we will study some properties of the operators $E_{R}$, $E_{I}$, and $\EE$.

Recall that $T^{-1}(\ve{u}) = (\sqrt{2}{u_{1}}, 2{u_{2}})$ (cf. \eqref{TRA}).

\begin{lemma}\label{Nl1}
Fix $\kk>0$.
\begin{enumerate}
    \item There exists $C > 0$ such that for every $\ve{v} \in \dot{H}^{1}(\R^{6}: \R) \times \dot{H}^{1}(\R^{6}: \R)$ satisfying
\begin{align}\label{ON1}
    (\ve{v}, T^{-1}(\Pi_{1}))_{\dot{H}^{1}} = (\ve{v}, T^{-1}(\Pi_{2}))_{\dot{H}^{1}} = (\ve{v}, T^{-1}(\Psi_{j}))_{\dot{H}^{1}} = 0
\end{align}
for $1 \leq j \leq 6$, then we have
\[
\<E_{R}\ve{v}, \ve{v}\> \geq C \|\ve{v}\|^{2}_{\dot{H}^{1}}.
\]

    \item There exists $C_{1} > 0$ such that for every $\ve{v} \in \dot{H}^{1}(\R^{6}: \R) \times \dot{H}^{1}(\R^{6}: \R)$ satisfying
\begin{align}\label{OF1}
    (\ve{v}, T^{-1}(\Xi_{1}))_{\dot{H}^{1}} = 0,
\end{align}
then we have
\[
\<E_{I}\ve{v}, \ve{v}\> \geq C_{1} \|\ve{v}\|^{2}_{\dot{H}^{1}}.
\]
Here, the functions $\Xi_{1}$, $\Pi_{1}$, $\Pi_{2}$ and $\Psi_{j}$ are given in \eqref{Ort223} and \eqref{piY}.
\end{enumerate}
\end{lemma}
\begin{proof}

Note that \eqref{ON1} is equivalent to the condition
\[
(T^{-1}(\ve{v}), \Xi_{1})_{\dot{H}^{1}} = (T^{-1}(\ve{v}), \Xi_{2})_{\dot{H}^{1}} = (T^{-1}(\ve{v}), \Psi_{j})_{\dot{H}^{1}} = 0.
\]
 Moreover, a simple calculation and Lemma~\ref{Apx11} show that
\[
\<E_{R}\ve{v}, \ve{v}\> = \tfrac{1}{2} \<L_{R}T^{-1}(\ve{v}), T^{-1}(\ve{v})\> \geq \tfrac{C}{2} \|T^{-1}(\ve{v})\|^{2}_{\dot{H}^{1}} \geq 
C\|\ve{v}\|^{2}_{\dot{H}^{1}}.
\]
This proves (i). The proof of (ii) is similar, and we omit the details.
\end{proof}

\begin{lemma}\label{Combi1}
Fix $\kk>0$.
\begin{enumerate}
    \item There exists a constant $C > 0$ such that for every $\ve{v} \in \dot{H}^{1}(\R^{6}: \R) \times \dot{H}^{1}(\R^{6}: \R)$ satisfying
    \begin{align}\label{FeN14}
        \F_{E}(\ve{v}, T(\Q)) = (\ve{v}, T(\Lambda \Q))_{\dot{H}^{1}} = (\ve{v}, T(\partial_{j}\Q))_{\dot{H}^{1}} = 0 \quad \text{for } 1 \leq j \leq 6,
    \end{align}
    the following inequality holds:
    \begin{align}\label{Coer:E}
        \<E_{R} \ve{v}, \ve{v}\> \geq C \|\ve{v}\|^{2}_{\dot{H}^{1}}.
    \end{align}

    \item There exists a constant $C > 0$ such that for every $\ve{v} \in \dot{H}^{1}(\R^{6}: \R) \times \dot{H}^{1}(\R^{6}: \R)$ satisfying
    \begin{align}\label{Orotvv}
        (\ve{v}, T(\Q_{1}))_{\dot{H}^{1}} = 0,
    \end{align}
    the following inequality holds:
    \begin{align}\label{CoerLI22}
        \<E_{I} \ve{v}, \ve{v}\> \geq C \|\ve{v}\|^{2}_{\dot{H}^{1}}.
    \end{align}
\end{enumerate}
\end{lemma}
\begin{proof}
With Lemma~\ref{Nl1} in hand, the proof follows the same lines as Lemmas~\ref{Coer11} and \ref{CoerLi11}.
\end{proof}

From Lemma~\ref{Combi1}  we obtain the following proposition.

\begin{proposition}\label{ApCo}
Fix $\kk>0$. Then, there exists $C>0$ such that for every $\ve{h}\in U^{\bot}$, we have
\[
\F_{E}(\ve{h})\geq C \|\ve{h}\|^{2}_{\dot{H}^{1}},
\]
where
\begin{equation*}
U^{\bot} := \left\{ \ve{h} \in \dot{H}^{1} \,\bigg|\, 
\begin{aligned}
&\F_{E}(T(\Q), \ve{h}) = (T(i\Q_{1}), \ve{h})_{\dot{H}^{1}} \\
&= (T(\Lambda \Q), \ve{h})_{\dot{H}^{1}} = (T(\partial_{j} \Q), \ve{h})_{\dot{H}^{1}} = 0, \\
&\quad j = 1, \ldots, 6
\end{aligned}
\right\}.
\end{equation*}
\end{proposition}

Since $\overline{\EE \ve{v}}= -\EE(\overline{\ve{v}})$, we infer that if $\lambda_{1}>0$ is an eigenvalue of the operator $\EE$ with the eigenfunction $e_{+}=(Y,Z)$, then $-\lambda_{1}$ is also an eigenvalue of $\EE$ with eigenfunction 
$e_{-}=\overline{e_{+}}=(\overline{Y}, \overline{Z})$. Writing $e_{1}=\RE e_{+}$ and $e_{2}=\IM e_{+}$. To show the existence of $e_{+}$, we must study the system 
\begin{equation}\label{SiE}
\begin{cases} 
E_{R} e_{1}=\lambda_{1}e_{2},\\
-E_{I} e_{2}=\lambda_{1}e_{1}.
\end{cases} 
\end{equation}
 By Lemma~\ref{Combi1}~(ii) we see that $E_{I}$ on $L^{2}$ with domain $H^{2}$ is nonnegative. Therefore,
as $E_{I}$ is self-adjoint, we obtain that $E_{I}$ has a unique square root $(E_{I})^{\frac{1}{2}}$ with domain $H^{1}$. 
With this in mind, consider the self-adjoint operator $\TT$ on $L^{2}$ 
with domain $H^{4}$,
\[
\TT=(E_{I})^{\frac{1}{2}}(E_{R})(E_{I})^{\frac{1}{2}}.
\]
Since
\[
\TT=(E_{I})^{2}-(E_{I})^{\frac{1}{2}}
\begin{bmatrix}
          2Q & 0 \\
           0& 0
  \end{bmatrix}
	(E_{I})^{\frac{1}{2}}
\]
and $Q$ is decreasing at infinity, it follows that $\TT$ is a relatively compact, self-adjoint, perturbation of $((\Delta)^{2}, (\kk\Delta)^{2})$.
Thus, by the Weyl theorem, we know that $\sigma_{\text{ess}}(\TT)=[0, \infty)$.

Suppose that there exists $\ve{g}\in H^{4}$ such that
\begin{equation}\label{egenT}
\TT \ve{g}=-\lambda^{2}_{1}\ve{g}.
\end{equation}
Writing
\[
e_{1}:=(E_{I})^{\frac{1}{2}}\ve{g} \quad \text{and} \quad e_{2}:=\frac{1}{\lambda_{1}}(E_{R})(E_{I})^{\frac{1}{2}}\ve{g}
\]
we obtain a solution to \eqref{SiE}, which implies the existence of the eigenfunction $e_{+}$. 

Now, to show the existence of $e_{+}$ we need to show that the operator $\TT$ has at least one negative eigenvalue $-\lambda^{2}_{1}$.
Indeed, we have the following result.

\begin{lemma}\label{NegativeL}
\[
\Pi(\TT):=\inf\left\{(\TT \ve{g}, \ve{g})_{L^{2}}, \ve{g}\in H^{4}, \|\ve{g}\|_{L^{2}}=1 \right\}<0.
\]
\end{lemma}
\begin{proof}
Writing $\Phi=T(\Lambda \Q)-\tfrac{(T(\Lambda \Q), T(i\Q_{1}))_{L^{2}}}{(T(\Q), T(i\Q_{1}))_{L^{2}}}T(\Q)$,
we see that $N_{0}:=\<E_{R}\Phi, \Phi \><0$ (recall that $\<E_{R}T(\Q),T(\Q)  \><0$).
Moreover, $(\Phi, T(i\Q_{1}))_{L^{2}}=0$ and $\Phi\in H^{2}$. In particular, we have
\[
((E_{I}+1)\Phi, T(i\Q_{1}))_{L^{2}}=0.
\]
Since
\[
\text{Ran}(E_{I})^{\bot}=\text{Ker}(E_{I})=\text{span}\left\{T(i\Q_{1})\right\},
\]
it follows that $(E_{I}+1)\Phi\in \overline{\text{Ran}(E_{I})}$. Thus, for $\epsilon>0$ (which will be chosen later) there exists $\ve{g}_{\epsilon}\in H^{2}$ such that
\begin{equation}\label{Gps}
\|E_{I}\ve{g}_{\epsilon}-(E_{I}+1)\Phi\|_{L^{2}}<\epsilon.
\end{equation}
We put $G_{\epsilon}:=(E_{I}+1)^{-1}E_{I}\ve{g}_{\epsilon} \in H^{2}$. From \eqref{Gps}, we infer that
\[
\|G_{\epsilon}-\Phi\|_{H^{2}}\leq \epsilon \|(E_{I}+1)^{-1}\|_{L^{2}\to L^{2}},
\]
which implies that there exists a constant $C_{0}>0$ such that
\[
\left|\<E_{R}G_{\epsilon}, G_{\epsilon}\>-\<E_{R}\Phi, \Phi \>\right|\leq C_{0}\epsilon.
\]
Choosing $\epsilon=\frac{-N_{0}}{2C_{0}}$, we see that $\<E_{R}G_{\epsilon}, G_{\epsilon}\><0$. 

Now, if $G=(E_{I})^{-\frac{1}{2}}G_{\epsilon}$, then 
\[
(\TT G, G)_{L^{2}}=(E_{R}G_{\epsilon}, G_{\epsilon})_{L^{2}}<0.
\]
Thus, we conclude that the operator $\TT$ has at least one negative eigenvalue.
\end{proof}

\begin{remark}\label{Rema11}
Notice that the operator $\EE$ is a compact perturbation of $(-i \Delta, -i {\kk}\Delta)$, thus we have 
$\sigma_{\text{ess}}(\EE)=i \R$. In particular, $\sigma(\EE)\cap (\R \setminus\left\{0\right\})$ contains only eigenvalues.
Moreover, by Lemma~\ref{NegativeL} we see that $\left\{\pm \lambda_{1}\right\}\subset \sigma(\EE)$. 
On the other hand, using the same argument developed in \cite[Subsection 7.2.2]{DuyMerle2009} we infer that 
$e_{\pm}\in \Sch(\R^{6})\times \Sch(\R^{6})$.
\end{remark}

\begin{remark}\label{Nzero} We have that $\F_{E}(e_{+}, e_{-})\neq 0$. Indeed, suppose instead that $\F_{E}(e_{+}, e_{-})= 0$. Consider the subspace
\[E=\mbox{span}\left\{T(i\Q_{1}), e_{+}, e_{-}, T(\Lambda \Q), T(\partial_{j} \Q): j=1,\ldots,6\right\},\] 
which is of co-dimension $10$. Then by Remark~\ref{PLf} we see that $\F_{E}(h)=0$ for all $h\in E$, which is a contradiction because $\F_{E}$ is positive on a co-dimension $9$ subspace (cf. Proposition~\ref{ApCo}).
\end{remark}

\begin{proof}[{Proof of Proposition~\ref{FCove}}] 
First we show that if $\ve{h}\in \tilde{G}^{\bot}$, then $\F_{E}(\ve{h})>0$. Suppose instead that there exists nonzero $\ve{g} \in \tilde{G}^{\bot}$ such that $\F_{E}(\ve{g})\leq 0$.
We now note that (cf. Remarks~\ref{PLf}~and~\ref{Nzero})
\begin{equation}\label{FEE}
\quad \F_{E}(e_{-})=0, \quad \F_{E}(e_{+})=0
\quad \text{and}\quad 
\F_{E}(e_{+}, e_{-})\neq 0.
\end{equation}
We set
\[
E_{-}:=\mbox{span}\left\{ T(i\Q_{1}), T(\Lambda \Q), e_{+}, \ve{g}, T(\partial_{j} \Q): j=1, \ldots,6\right\}.
\]
By \eqref{FEE} and Remarks~\ref{PLf} we get $\F_{E}(\ve{h})\leq 0$ for all $\ve{h}\in E_{-}$.
Recall $\F_{E}(e_{+}, e_{-})\neq 0$. Since $T(i\Q_{1})$, $T(\Lambda \Q)$, $\ve{g}$ and $T(\partial_{j} \Q)$ are orthogonal in the real Hilbert space $\dot{H}^{1}$ one can show that $\mbox{dim}_{\R} E_{-}=10$. However, Lemma~\ref{Combi1} shows that $\F_{E}$ is definite positive on a co-dimension $9$ subspace of $\dot{H}^{1}$, which is a contradiction. Therefore, $\F_{E}(h)>0$ for all $\ve{h}\in \tilde{G}^{\bot}$.

Finally, as $T(\Q)$ is decreasing at infinity, by compactness, the coercivity follows on $\tilde{G}^{\bot}$.
\end{proof}

\begin{proof}[{Proof of Lemma~\ref{SpecLL}}] By Remark~\ref{Rema11},
it remains to show that $\sigma(\EE)\cap (\R \setminus\left\{0\right\})=\left\{-\lambda_{1}, \lambda_{1}\right\}$.
Indeed, by contradiction, assume that there exists $\ve{f}\in H^{2}$ such that
\[
\EE \ve{f}=-\lambda_{0}\ve{f},
\]
with $\lambda_{0}\in \R\setminus\left\{0, -\lambda_{1}, \lambda_{1}\right\}$. Since $\F_{E}(\EE \ve{g},\ve{h})=-\F_{E}(\ve{g}, \EE \ve{h})$ we see that
\[
(\lambda_{1}+\lambda_{0})\F_{E}(\ve{f}, e_{+})=(\lambda_{1}-\lambda_{0})\F_{E}(\ve{f}, e_{-})=0
\quad \text{and}\quad 
\lambda_{0}\F_{E}(\ve{f},\ve{f})=-\lambda_{0}\F_{E}(\ve{f},\ve{f}),
\]
i.e.,
\[
\F_{E}(\ve{f}, e_{+})=\F_{E}(\ve{f}, e_{-})=\F_{E}(\ve{f},\ve{f})=0.
\]
We write

\[
\ve{f} = i\beta_{0}T(\Q_{1}) + \sum_{j=1}^{6} \beta_{j}T(\partial_{j} \Q) + \alpha T(\Lambda \Q) + g,
\]
\[
\text{with} \,\,
\beta_{0} = \frac{(\ve{f}, T(i\Q_{1}))_{\dot{H}^{1}}}{\|T(i\Q_{1})\|^{2}_{\dot{H}^{1}}}, \quad
\beta_{j} = \frac{(\ve{f}, T(\partial_{j} \Q))_{\dot{H}^{1}}}{\|T(\partial_{j} \Q)\|^{2}_{\dot{H}^{1}}}, \quad
\alpha = \frac{(\ve{f}, T(\Lambda \Q))_{\dot{H}^{1}}}{\|T(\Lambda \Q)\|^{2}_{\dot{H}^{1}}},
\]
where $g\in \tilde{G}^{\bot}$. From Remark~\ref{PLf} we get $\F_{E}(\ve{g}, \ve{g})=\F_{E}(\ve{f}, \ve{f})=0$. Then,
 Proposition~\ref{FCove} implies 
\[
\|\ve{g}\|^{2}_{H^{1}}\lesssim \F_{E}(\ve{g})=0.
\]
Therefore, $\ve{g}=0$ and 
\[
\lambda_{0}\ve{f}=\EE \ve{f}=\EE\ve{g}=0,
\]
which is a contradiction. This completes the proof of Lemma~\ref{SpecLL}.
\end{proof}

\begin{remark}\label{KerLR}
As a consequence of Proposition~\ref{FCove}, we have that 
\begin{align}\label{Ker11}
	\text{Ker}(\EE) = \text{span}\left\{T(i\Q_{1}), T(\Lambda Q), T(\partial_{j}\Q) : j = 1, \ldots, 6\right\}.
\end{align}
Indeed, first notice that $\text{span}\left\{T(i\Q_{1}), T(\Lambda Q), T(\partial_{j}\Q) : j = 1, \ldots, 6\right\} \subseteq \text{Ker}(\EE)$. Now, if the dimension of $\text{Ker}(\EE)$ were strictly higher than $8$, then there would exist a nonzero $\ve{f} \in L^{2}$ such that $\EE \ve{f} = 0$. 
Writing 
\[
 \ve{f}^{\prime} = \ve{f} - \tfrac{(\ve{f}, T(i\Q_{1}))_{\dot{H}^{1}}}{\|T(\Q_{1})\|^{2}_{\dot{H}^{1}}}T(i\Q_{1})
- \tfrac{(\ve{f}, T(\Lambda\Q))_{\dot{H}^{1}}}{\|T(\Lambda\Q)\|^{2}_{\dot{H}^{1}}}T(\Lambda\Q)
- \sum^{6}_{j=1}\tfrac{(\ve{f}, T(\partial_{j}\Q))_{\dot{H}^{1}}}{\|T(\partial_{j}\Q)\|^{2}_{\dot{H}^{1}}}T(\partial_{j}\Q) \neq 0,
\]
we see that $\EE \ve{f}^{\prime} = 0$, 
\[
(\ve{f}^{\prime}, T(i\Q_{1}))_{\dot{H}^{1}} = (\ve{f}^{\prime}, T(\Lambda\Q))_{\dot{H}^{1}} = (\ve{f}^{\prime}, T(\partial_{j}\Q))_{\dot{H}^{1}} = 0,
\]
and
\[
\F_{E}(e_{\pm}, \ve{f}^{\prime}) = \pm\tfrac{1}{\lambda_{1}}\F_{E}(\EE e_{\pm}, \ve{f}^{\prime}) =
\mp\tfrac{1}{\lambda_{1}}\F_{E}(e_{\pm}, \EE \ve{f}^{\prime}) = 0.
\]
Therefore, $\ve{f}^{\prime} \in \tilde{G}^{\bot}$. Since $\F_{E}(\ve{f}^{\prime}) = 0$, Proposition~\ref{FCove} implies that $\ve{f}^{\prime} = 0$,
which is a contradiction. 

In particular, by \eqref{Ker11}, we deduce that 
\[
\text{Ker}(E_{R}) = \text{span}\left\{T(\Lambda Q), T(\partial_{j}\Q) : j = 1, \ldots, 6\right\}.
\]
Moreover, Lemma~\ref{Combi1} implies that the operator $E_{R}$ has only one negative eigenvalue.
\end{remark}


\end{document}